%% file: staircase.tex
\documentclass[11pt]{article}
\RequirePackage[colorlinks,citecolor=blue,urlcolor=blue,linkcolor=blue]{hyperref}
\usepackage{url}

\input{preamble}

\title{A Riemannian low-rank method for optimization over semidefinite matrices with block-diagonal constraints}

\author{
Nicolas Boumal \\ %
Inria \& D.I. Sierra, UMR 8548, Ecole Normale Sup\'erieure, Paris, France\\
\texttt{nicolas.boumal@ens.fr} \\
}

\begin{document}

\maketitle

\begin{abstract}
We propose a new algorithm to solve optimization problems of the form $\min f(X)$ for a smooth function $f$ under the constraints that $X$ is positive semidefinite and the diagonal blocks of $X$ are small identity matrices. Such problems often arise as the result of relaxing a rank constraint (lifting).
In particular, many estimation tasks involving phases, rotations, orthonormal bases or permutations fit this framework, and so do certain relaxations of combinatorial problems such as Max-Cut.
The proposed algorithm exploits the facts that~(1) such formulations admit low-rank solutions, and~(2) their rank-restricted versions are smooth optimization problems on a Riemannian manifold. Combining insights from both the Riemannian and the convex geometries of the problem, we characterize when second-order critical points of the smooth problem reveal KKT points of the semidefinite problem. In particular, we bound the ranks that need to be considered, deterministically. Comparison against state of the art, mature software shows that, on certain interesting problem instances, what we call the staircase method is orders of magnitude faster, is more accurate and scales better.
Code is available.
\end{abstract}

\clearpage
\section{Introduction}

This paper considers the generic problem of estimating matrices $Y_1, \ldots, Y_m \in \Rdp$, $d \leq p$, with orthonormal rows, that is, such that $Y_i^{}Y_i\transpose = I_d$ (identity of size $d$) for all $i$. We further focus on problems where only \emph{relative} information is available, that is, information about $Y_i^{}Y_j\transpose$ for some of the pairs $(i,j)$ is available, but there is no information about individual $Y_i$'s. %
As will be detailed below, particular cases of this come up in a number of applications. For example, when $p = d = 1$, the variables $Y_i$ reduce to $\{\pm 1 \}$ and the products $Y_i^{}Y_j\transpose$ indicate whether $Y_i$ and $Y_j$ have the same sign or not, allowing to model certain combinatorial problems. When $d = 1, p > 1$, the variables reduce to unit-norm vectors, and the products correspond to inner products between them, allowing to model correlations and proximity on a sphere. Finally, when $d = p > 1$, the matrices are orthogonal, and the products $Y_i^{}Y_j\transpose = Y_i^{}Y_j^{-1}$ represent relative orthogonal transformations, such as rotations, reflections and permutations.

For ease of notation, we stack the orthonormal matrices on top of each other to form $Y \in \Rnp$, $n = md \geq p$. Then, $X = YY\transpose$ is a block matrix whose block $X_{ij} \in \Rdd$ corresponds to the relative product $Y_i^{}Y_j\transpose$. Define the (transposed) \emph{Stiefel manifold} as
\begin{align*}
\St(d, p) = \{ Z \in \Rdp : ZZ\transpose = I_d \},
\end{align*}
and the set of $Y$'s obtained by stacking as
\begin{align}
\Stdpm & = \left\{ Y \in \reals^{md\times p} : Y\transpose = \begin{pmatrix}
Y_1\transpose & Y_2\transpose & \cdots & Y_m\transpose
\end{pmatrix} \textrm{ and } Y_1, \ldots, Y_m \in \St(d, p) \right\} \nonumber \\
	& = \left\{ Y\in\Rnp : (YY\transpose)_{ii} = I_d \textrm{ for } i = 1\ldots m \right\}.
	\label{eq:Stdpm}
\end{align}
This paper is concerned with solving optimization problems of the form
\begin{align}
\tag{$\textrm{RP}_p$}
\label{eq:RP}
\min_{Y\in\Rnp} \ \ g(Y) = f(YY\transpose), \quad %
\textrm{subject to} \quad Y \in \Stdpm,
\end{align}
with twice continuously differentiable cost $f \colon \Snn \to \reals$ defined over the symmetric matrices. Here, $g$ is for example the negative likelihood of $Y$ with respect to available data. %
The restriction that $g(Y)$ be only a function of $YY\transpose$ encodes the property that only relative information is available, through $(YY\transpose)_{ij} = Y_i^{}Y_j\transpose$. This induces invariance of the cost under right-action of the orthogonal group. Indeed, $g(YQ) = g(Y)$ for any orthogonal matrix $Q$ of size $p$. Thus, solutions of~\eqref{eq:RP} are only defined up to this group action.

Problem~\eqref{eq:RP} is computationally hard. In particular, for $d = p = 1$ and linear $f$, it covers the NP-hard Max-Cut problem~\citep{goemans1995maxcut}. Following that and other previous work~\citep{beck2007quadratic,singer2011eigen,bandeira2013approximating}, we consider a relaxation through the following observation. For all $Y\in\Stdpm$, the matrix $X = YY\transpose$ is positive semidefinite, its diagonal blocks $X_{ii}$ are identity matrices $I_d$, and it has rank at most $p$. Conversely, any matrix $X$ with those properties can be factored as $YY\transpose$ with $Y\in\Stdpm$. In other words, problem~\eqref{eq:RP} is equivalent to optimizing $f$ over the convex set
\begin{align}
	\calC & = \left\{ X \in \Snn : X \succeq 0 \textrm{ and } X_{ii} = I_d \textrm{ for } i = 1\ldots m \right\},
	\label{eq:calC}
\end{align}
with the additional constraint $\rank(X) \leq p$. As often, the rank constraint is the culprit. Indeed, continuing with the Max-Cut example (linear $f$), optimization over $\calC$ without the rank constraint is a \emph{semidefinite program}, which can be solved to arbitrary precision in polynomial time~\citep{vandenberghe1996semidefinite}.

This is motivation to study the relaxation obtained by ignoring the rank constraint (for linear $f$, it is also the dual of the dual of~\eqref{eq:RP}):
\begin{align}
\tag{P}
\label{eq:P}
\min_{X\in\Rnn} \ \ f(X), \quad \textrm{subject to} \quad X\in\calC.
\end{align}
The optimal cost of~\eqref{eq:P} is a lowerbound on that of~\eqref{eq:RP}. Furthermore, if~\eqref{eq:P} admits a solution $X = YY\transpose$ with $Y\in\Rnp$, that is, a solution of rank at most $p$, then $Y$ is a solution of~\eqref{eq:RP}. When that is not the case, a higher-rank solution $X$ may still be projected to a (hopefully good) initial guess for the nonconvex problem~\eqref{eq:RP}. See \citep{naor2013efficient,bandeira2013approximating} for a discussion of approximation results related to these projections.
The price to pay is that $\calC$ is much higher dimensional than $\Stdpm$: this is called a \emph{lift}~\citep{beck2007quadratic}.

For linear $f$, solutions of~\eqref{eq:P} can of course be computed using standard SDP solvers, such as interior point methods (IPM). Unfortunately, as demonstrated in Section~\ref{sec:linear}, IPM's do not scale well. The main reason for it is that, as the name suggests, IPM's iterate inside the interior of the search space $\calC$. The latter is formed by full-rank, dense matrices of size $n$:
this quickly becomes unmanageable.

The full-rank operations seem even more wasteful considering that, still for linear $f$, problem~\eqref{eq:P} always admits a solution of rank at most
\begin{align}
	p^* = \frac{\sqrt{1 + 4md(d+1)} - 1}{2} < (d+1)\sqrt{m} \ll n.
	\label{eq:pstar}
\end{align}
Indeed, this follows a general result of \citet{shapiro1982rank}, \citet{barvinok1995problems} and \citet{pataki1998rank} regarding extreme points of \emph{spectrahedra},\footnote{The name spectrahedron for the search space of a semidefinite program echoes the name polyhedron for the search space of a linear program.} that is, intersections of the positive semidefinite cone with an affine subspace---the geometry of $\calC$ is discussed in Section~\ref{sec:geometryconvex}. This prompted \citet{sdplr,burer2005local} to propose SDPLR, a generic SDP solver which exploits the low-rank phenomenon. Applying SDPLR to our problem amounts to computing a local minimizer $Y$ of~\eqref{eq:RP} for some small $p$, using classical nonlinear optimization algorithms and penalizing for the constraints in a Lagrangian way. Then, $p$ is increased as needed until $YY\transpose$ can be certified as a solution to the SDP.

SDPLR is powerful and generic, and the theory accompanying the algorithm brings great insight into the problem. But it also has some downsides we want to improve on in the context of~\eqref{eq:P}. First, it is not an easy matter to guarantee convergence to (even local) optimizers in the nonlinear subproblems. Furthermore, since constraints are enforced by penalization, they are not accurately satisfied by the returned solution. Finally, we would like to allow for nonlinear $f$. Nevertheless, Section~\ref{sec:linear} shows SDPLR improves significantly upon IPM's.

\citet{journee2010low} build upon SDPLR, observing that certain SDP's harbor an elegant Riemannian geometry that can be put to good algorithmic use. In particular, they cover what here corresponds to the case $d = 1$ and observe that, as remains true for $d > 1$, \eqref{eq:RP} is an optimization problem on a smooth space: $\Stdpm$ is a \emph{Riemannian manifold}---this geometry is detailed in Section~\ref{sec:geometryriemannian}. Allowing for smooth nonlinear $f$, they apply essentially the SDPLR machinery, replacing the nonlinear programming algorithms for~\eqref{eq:RP} by Riemannian optimization algorithms~\citep{AMS08}. These algorithms exploit the smooth structure of the nonconvex search space, resulting in constraint satisfaction up to numerical accuracy, as well as notable speedups.

As a further refinement, \citet{journee2010low} address the invariance of $f$ under orthogonal group action. Instead of optimizing $f(YY\transpose)$ over $\Stdpm$, they optimize over the quotient space $\Stdpm_* / \! \sim$, where $\sim$ is an equivalence relation defined over $\Stdpm_*$ (the full-rank elements of $\Stdpm$) by $Y\sim \tilde Y \Leftrightarrow YY\transpose = \tilde Y \tilde Y\transpose$. The advantage is that this quotient space, which is still a smooth Riemannian manifold, is now one-to-one with the rank-$p$ matrices in $\calC$. Unfortunately, the geometry breaks down at rank-deficient $Y$'s (to see this, notice that equivalence classes of different rank have different dimension; see also Figure~\ref{fig:geometry}). The breakdown is problematic since, as will become clear, it is desirable to converge to rank-deficient $Y$'s. Furthermore, that paper too asks for computation of local optimizers of the subproblems, which, on Riemannian manifolds too, is a difficult task.

In both~\citep{burer2005local} and~\citep{journee2010low}, one of the keys to practical efficiency is (well-justified) optimism: \eqref{eq:RP} is first solved for small values of $p$, and $p$ is increased only as needed. In both papers, it is observed that, in practice, it often suffices to reach $p$ just above the rank of the target solution of~\eqref{eq:P}, which may be quite small; but there is no theory to confirm this. We do not prove such a strong result either, but we give some nontrivial, deterministic bounds on ``how high one must lift'', refining certain results of~\citep{burer2005local}.

\subsection{Contribution}

In this paper, we describe the Riemannian geometry of $\Stdpm$ in order to frame~\eqref{eq:RP} as a Riemannian optimization problem. We use existing algorithms~\citep{AMS08} and the Manopt toolbox~\citep{manopt} to compute critical points of~\eqref{eq:RP}, that is, points where the (Riemannian) gradient of the cost function vanishes. In practice, those algorithms tend to converge to second-order critical points, that is, points where the (Riemannian) Hessian is also positive semidefinite, because all other critical points are unstable fixed points of the iteration.

For $p > d$, $\Stdpm$ is a connected,\footnote{For $p=d$, $\Stdpm$ has $2^m$ disconnected components, because the orthogonal group has two components: matrices with determinant $+1$ and $-1$. This is a strong incentive to relax at least to $p = d+1$.} compact and smooth space. Since we further assume sufficient smoothness in $f$ too, this makes for a nice problem with no delicate limit cases to handle. Furthermore, Riemannian optimization algorithms iterate on the manifold directly: all iterates satisfy constraints up to numerical accuracy.

We then turn our attention to computing Karush-Kuhn-Tucker (KKT) points for~\eqref{eq:P}. These are points that satisfy first-order necessary optimality conditions. If $f$ is convex, the conditions are also sufficient.
Our goal is to compute KKT points
via the computation of second-order critical points of~\eqref{eq:RP}, which is lower-dimensional. A key property that makes this possible is the availability of an explicit dual matrix $S(X)$~\eqref{eq:S} which intervenes in both sets of conditions. 

Using this dual matrix, we show that rank-deficient second-order critical points $Y$ reveal KKT points $X = YY\transpose$. Furthermore, when a computed second-order critical point is full rank, it is shown how to use it as a warm-start for the computation of a second-order critical point of~\eqref{eq:RP} with a larger value of $p$. It is guaranteed that if $p$ is allowed to grow up to $n$, then all second-order critical points reveal KKT points, so that the procedure terminates. This is formalized in Algorithm~\ref{algo:staircase}, which we call the \emph{Riemannian Staircase}, as it lifts~\eqref{eq:RP} to~\eqref{eq:P} step by step, instead of all at once.

The above points rest extensively on work discussed earlier in this introduction~\citep{sdplr,burer2005local,journee2010low}, and improve upon those along the lines announced in the same. In particular, we do not require the computation of local optimizers of~\eqref{eq:RP}, and we avoid the geometry breakdown tied to the quotient approach in~\citep{journee2010low}. We also stress that the latter reference only covers $d = 1$, and SDPLR only covers linear $f$.

We further take particular interest in understanding how large $p$ may grow in the staircase algorithm. We view this part as our principal theoretical contribution. This investigation calls for inspection of the convex geometry of $\calC$, with particular attention to its faces and their dimension. To this effect, we use results by~\citet{pataki1998rank} to describe the face of $\calC$ which contains a given $X$ in its relative interior, and we quote the lower-bound on the dimension of that face as a function of $\rank(X)$. We further argue that this bound is almost always tight, and we give an essentially tight upper bound on the dimension of a face, generalizing a result of~\citet{laurent1996facial} to $d > 1$.

Using this facial description of $\calC$, we establish that for strongly concave $f$, for $p > p^*$~\eqref{eq:pstar}, all second-order critical points of~\eqref{eq:RP} reveal KKT points. Also, for concave $f$, we show the same for $p > \frac{d+1}{d+3}n$ (Corollary~\ref{cor:sufficientp}), and argue that $p > p^*$ is sufficient under an additional condition we believe to be mild. Hence,

\begin{quote}
	\emph{For linear $f$, above a certain threshold for $p$, all second-order critical points of~\eqref{eq:RP} are global optimizers.}
\end{quote}

The condition is stronger than the one proposed in~\citep{burer2005local}, and the statement is about second-order critical points, rather than about local optimizers of~\eqref{eq:RP}. There are no similar results for convex $f$, as then solutions can have any rank.

We close the paper with numerical experiments showing the efficiency of the staircase algorithm to solve~\eqref{eq:P} on certain synchronization problems involving rotations and permutations, as compared to IPM's and SDPLR.

Note that, up to a linear change of variable, problem~\eqref{eq:P} also encompasses constraints of the form $X_{ii} = B_i$ where each $B_i$ is positive definite. We assume all diagonal blocks have identical size $d$ as this simplifies exposition, but the proposed method can easily accommodate inhomogeneous sizes, and many of the developments go through for complex matrices as well.

\subsection{Applications}

Problem~\eqref{eq:RP} and its relaxation~\eqref{eq:P}
appear in numerous applications. Many of those belong to the class of \emph{synchronization problems}, which consist in estimating group elements from measurements of pairwise ratios. Further applications are also described, e.g., in~\citep{singer2010angular,naor2013efficient,bandeira2013approximating}.

\paragraph{Combinatorial problems} can be modeled in~\eqref{eq:RP} with $d = p = 1$. A seminal example is \textbf{Max-Cut}: the problem of clustering a graph in two classes, so as to maximize the sum of weights of edges joining the two classes. The cost $f$ is linear, determined by the graph's adjacency matrix. Its relaxation to~\eqref{eq:P} is the subject of an influential analysis by~\citet{goemans1995maxcut}, which helped popularize the type of lifts considered here. See \citep[eq.\,(3)]{das2015sdhap} for a recent application of Max-Cut to genomics. The same setup, but with different linear costs, appears in the \textbf{stochastic block model}~\citep{abbe2014exact}, in \textbf{community detection}~\citep{cucuringu2014synchronization}, in \textbf{maximum a posteriori (MAP) inference in Markov random fields} with binary variables and pairwise interactions~\citep{frostig2014simple} and in \textbf{robust PCA}~\citep[Alg.\,1]{mccoy2011robustpca}. All of these study the effects of the relaxation on the final outcome, mostly under random data models. Their linear cost matrices are often structured (sparse or low-rank), which is easily exploited here.

\paragraph{Spherical embeddings} is the general problem of estimating points on a sphere in $\Rp$, and appears notably in machine learning for classification~\citep{wilson2010spherical} and in the fundamental problem of \textbf{packing spheres on a sphere}~\citep{cohn2007universally}. It is modeled by~\eqref{eq:RP} with $d = 1, p > 1$. The same setup also models \textbf{correlation matrix completion and approximation}~\citep{grubisic2007lowrankcorrelation}. In the latter, an algorithm to solve~\eqref{eq:P} is proposed, which inspired~\citep{journee2010low}, which inspired this work.

\paragraph{Synchronization of rotations} is the problem of estimating $m$ rotation matrices (orthogonal matrices with determinant 1, to exclude reflections), based on pairwise relative rotation measurements. It is modeled in~\eqref{eq:RP} with $d = p > 1$ (often, 2 or 3) and comes up in~\textbf{structure from motion}~\citep{arienachimson2012global}, \textbf{pose graph estimation}~\citep{calafiore2015pose}, \textbf{global registration}~\citep{chaudhury2013global}, the~\textbf{generalized Procrustes problem}~\citep{tenberge1977procrustes} and \textbf{simultaneous localization and mapping (SLAM)}~\citep{carlone2015slam}. It serves in global camera path estimation~\citep{bourmaud2015motion}, scan alignment~\citep{bonarrigo2011enhanced,wang2012LUD}, and sensor network localization and the molecule problem~\citep{cucuringu2011sensor,cucuringu2011eigenvector}. In many of these problems, translations must be estimated as well, and it has been shown in practical contexts that rotations and translations are best estimated separately~\citep[Fig.\,1]{carlone2015slam}. Here,~\eqref{eq:RP} can easily accommodate the determinant constraint: it comes down to picking one of the connected components of $\Stdpm$, as in~\citep{boumal2013MLE}. The relaxation~\eqref{eq:P} ignores this, though; see~\citep{saunderson2014semidefinite} for relaxations which explicitly model this difference (at additional computational cost). The problem of estimating orthogonal matrices appears notably in the \textbf{noncommutative little Grothendieck problem}~\citep{naor2013efficient,bandeira2013approximating}. In the latter, the relaxation~\eqref{eq:P} with linear $f$ is called \textbf{Orthogonal-Cut}, and its effect on~\eqref{eq:RP} is analyzed. The same relaxation with a nonsmooth cost, for robust estimation, is proposed and analyzed in~\citep{wang2012LUD}. See also~\citep{arrigoni2014robust} for another robust formulation of the same problem, based on low-rank--plus--sparse modeling.

\paragraph{The common lines problem in Cryo-EM} is an important biomedical imaging instance of~\eqref{eq:RP}, where orthonormal matrices are to be estimated with $d = 2, p = 3$~\citep{wang2012orientation}.

\paragraph{Phase synchronization and recovery} can be modeled with $p = d = 2$ (as phases are rotations in $\reals^2$). It is sometimes attractive to model phases as unit-modulus complex numbers instead, as is done in~\citep{singer2010angular} for phase synchronization, with the same SDP relaxation. This can be used for \textbf{clock synchronization}. See~\citep{bandeira2014tightness} for a study of the tightness of this SDP, and~\citep{cucuringu2015syncrank} for an application to \textbf{ranking}. The \textbf{Phase-Cut} algorithm for phase recovery uses the same SDP~\citep{waldspurger2012phase}, with a different linear cost.
While not explicitly treated, many of the results in this paper extend to the complex case.

\subsection{Related work}

Problem~\eqref{eq:RP} is an instance of optimization on manifolds~\cite{AMS08}. Optimization over orthonormal matrices is also studied in, e.g.,~\citep{edelman1998geometry,wen2013orthogonality}. Being equivalent to~\eqref{eq:P} with a rank constraint, \eqref{eq:RP} also falls within the scope of optimization over matrices with bounded rank~\citep{uschmajew2014convergence,mishra2011low}, where the latter is also an extension of~\citep{journee2010low}. %
The particular case of optimization over bounded-rank positive semidefinite matrices with linear constraints was already addressed in~\citep{tarazaga1993optimization}. The same without positive semidefiniteness constraint is studied recently in~\citep{liu2015snig}, also with a discussion of global optimality of second-order critical points. With a linear cost $f$, problem~\eqref{eq:RP} (which then has a quadratic cost $g$) is a subclass of quadratically constrained quadratic programming (QCQP). QCQP's and their SDP relaxations have been extensively studied, notably in~\citep{nemirovski2007quadratic,so2009improved}, with particular attention to approximation ratios. For~\eqref{eq:P}, these approximation ratios can be found in~\citep{bandeira2013approximating}.

In part owing to the success of~\eqref{eq:P} with linear $f$ in adequately solving a myriad of hard problems, there has been strong interest in developing fast, large-scale SDP solvers. The present paper is one example of such a solver, restricted to the class of problems~\eqref{eq:P}. SDPLR is a more generic such solver~\citep{sdplr,burer2005local}. See also~\citep{tu2014practical} for a review, and~\citep{desa2014global} for a recent low-complexity example with precise convergence results, but which does not handle constraints.

Much of this paper is concerned with characterizing the rank of solutions of~\eqref{eq:P}, especially with respect to how large $p$ must be allowed to grow in~\eqref{eq:RP} to solve~\eqref{eq:P}. There is also considerable value in determining under what conditions~\eqref{eq:P} admits solutions of the desired rank for a specific application, that is: when is the relaxation tight? This question is partially answered in~\citep{bandeira2014tightness} for the closely related phase synchronization problem, under a \emph{stochastic} model for the data. See~\citep{abbe2014exact} for a  proof in the stochastic block model, and~\citep{amelunxen2014edge} for a study of phase transitions in random convex programs. There also exist \emph{deterministic} tightness results, typically relying on special structure in a graph underlying the problem data. See for example~\citep{sojoudi2014exactness,saunderson2014semidefinite,sagnol2011rankone}. See also Appendix~\ref{sec:cycletight} for a deterministic proof of tightness in the case of single-cycle synchronization of rotations. The proof rests on the availability of a closed-form expression for the dual matrix $S$~\eqref{eq:S}, and for the solution to be certified. With the same ingredients, it is easy to show, for example, that~\eqref{eq:P} is tight for Max-Cut when the graph is bipartite.

Semidefinite relaxations in the form of~\eqref{eq:P} with additional constraints have also appeared in the literature. In particular, this occurs in estimation of rotations, with explicit care for the determinant constraints: \citet{saunderson2014semidefinite} explicitly constrain off-diagonal blocks to belong to the convex hull of the rotation group; this is not necessary for the orthogonal group---see Proposition~\ref{prop:Xijconvhull}. Similarly, for synchronization of permutations in joint shape matching, off-diagonal blocks are restricted to be doubly stochastic~\citep{chen2014near,huang2013consistent}. Finally, in recent work, \citet{bandeira2015nonuniquegames} study a more powerful class of synchronization problems with additional linear constraints of various forms. An example with an additional \emph{nonlinear} constraint appears in~\citep{wang2012LUD}, which imposes an upperbound on the spectral norm of $X$. All of these are motivation to generalize the framework studied here, in future work.

We mention in passing that the MaxBet and MaxDiff problems~\citep{tenberge1988maxbetmaxdiff} do not fall within the scope of this paper. Indeed, although they also involve estimating orthonormal matrices as in~\eqref{eq:RP}, their cost function has a different type of invariance, which would also lead to a different type of relaxation. %

\subsection{Notation}

The size parameters obey $1 \leq d \leq p \leq n = md$. Matrices $A\in\Rnn$ are thought of as block matrices with blocks of size $d\times d$. Subscript indexing such as $A_{ij}$ refers to the block on the \ith row and \jth column of blocks, $1 \leq i, j \leq m$. For $Z\in\Rnp$, $Z_i$ refers to the \ith slice of size $d\times p$, $1 \leq i \leq m$. The Kronecker product is written $\otimes$ and $\vecc$ vectorizes a matrix by stacking its columns on top of each other. A real number $a$ is rounded down as $\floor{a}$. The operator norm $\opnormsmall{A} = \sigma_{\textrm{max}}(A)$ is the largest singular value of a matrix, and its Frobenius norm $\frobnormsmall{A}$ is the $\ell_2$-norm of $\vecc(A)$. $\Snn$ is the set of symmetric matrices of size $n$, and $A\succeq 0$ means $A\in\Snn$ is positive semidefinite. $\symm{A} = (A+A\transpose)/2$ extracts the symmetric part of a matrix. $\Od$ is the group of orthogonal matrices of size $d$. $\ker \calL$ denotes the null-space, or kernel, of a linear operator.

\begin{figure}[p]
	\begin{center}
		\includegraphics[width=0.8\textwidth]{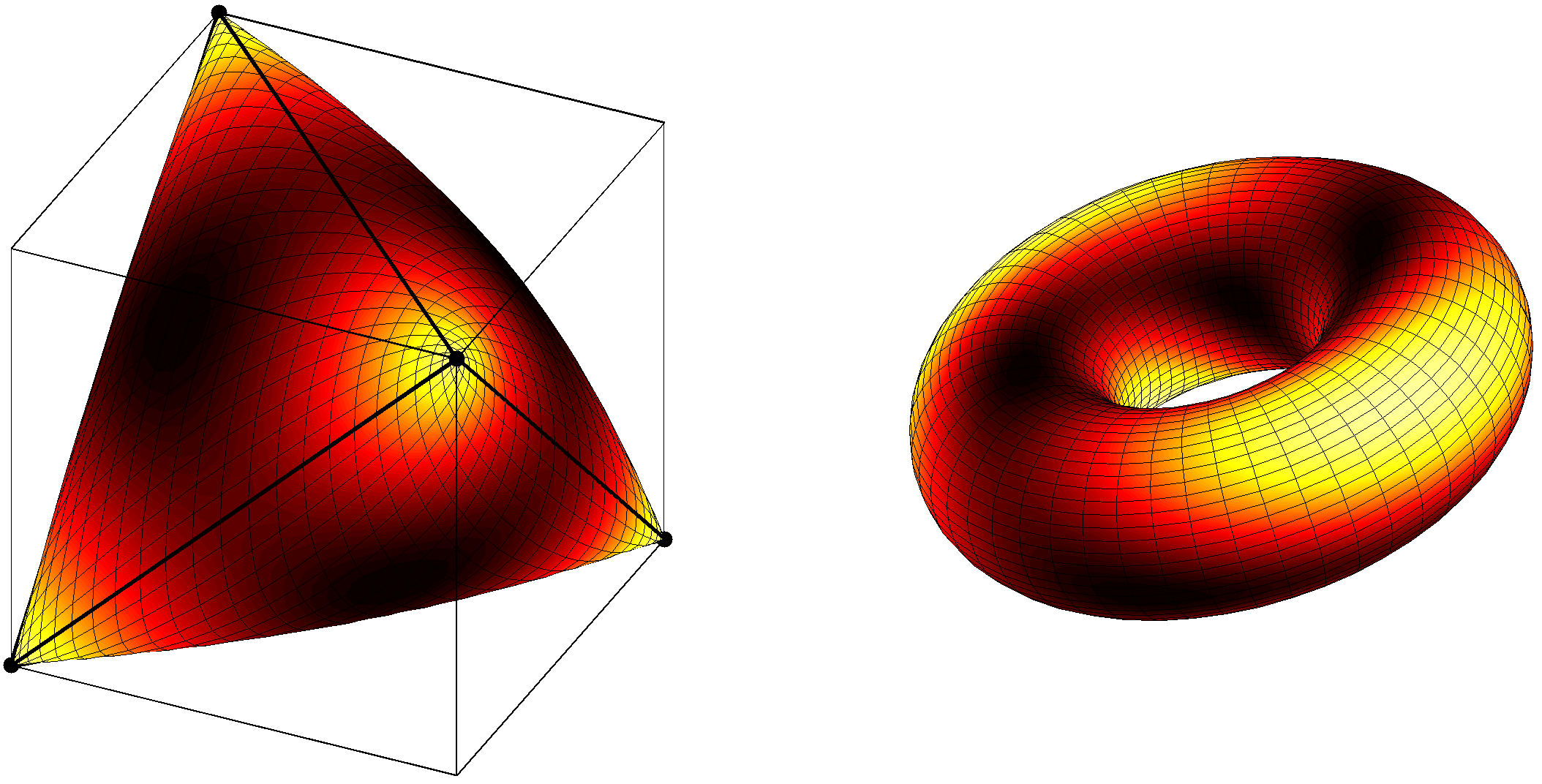}
	\end{center}
	\caption{
		(Left) For $m = 3, d = 1$, the set $\calC$~\eqref{eq:calC} contains all positive semidefinite matrices of the form $X = [1, a, b ; a, 1, c ; b, c, 1]$. It is here represented in coordinates $(a,b,c)$. The interior of the shell contains the rank-3 matrices; the four extreme points (black dots) are the rank-1 matrices; and the remainder of the boundary is the (smooth) set of rank-2 matrices. That smooth geometry breaks down at the rank-1 matrices. Note that for $d>1$, extreme points of rank $d$ are no longer isolated.
		(Right) For $m = 3, d = 1, p = 2$, the set $\Stdpm$~\eqref{eq:Stdpm} parameterizes the matrices of rank at most 2 in $\calC$ (redundantly). In this case, $\Stdpm$ corresponds to a product of three circles in 2D: $Y = [\cos \alpha_1, \sin \alpha_1 ; \cos \alpha_2, \sin \alpha_2 ; \cos \alpha_3, \sin \alpha_3] \in \Stdpm$. One of these degrees of freedom is redundant, because the factorization of $X \in \calC$ as $YY\transpose$ is not unique. The figure represents the remaining degrees of freedom after fixing $\alpha_1 = 0$. Notice how accepting the redundancy in the parameterization allows for a smooth representation of the nonsmooth set of bounded rank matrices in $\calC$.
		Color codes for $\frobnormsmall{X} = \frobnormsmall{YY\transpose}$.}
	\label{fig:geometry}
\end{figure}

\section{Geometry} \label{sec:geometry}

Both search spaces of~\eqref{eq:RP} and~\eqref{eq:P} enjoy rich geometry, which leads to efficient analytical and numerical tools for the study of these optimization problems. The former is a smooth Riemannian manifold, while the latter is a compact convex set. Figure~\ref{fig:geometry} depicts the two.

\subsection{Smooth geometry of the rank-restricted search space}
\label{sec:geometryriemannian}

Endow $\Rnp$ with the classical Euclidean metric $\inner{U_1}{U_2} = \Trace(U_1\transpose U_2^{})$, corresponding to the Frobenius norm: $\sqfrobnorm{U} = \inner{U}{U}$. We view the search space of~\eqref{eq:RP} as a submanifold of $\Rnp$ and endow it with the Riemannian submanifold geometry~\citep{AMS08}. First, define a linear operator $\sbdop \colon \Rnn \to \Snn$ which symmetrizes diagonal blocks and zeroes out all other blocks:
\begin{align}
	\sbd{M}_{ij} & = \begin{cases}
	\frac{M_{ii}^{} + M_{ii}\transpose}{2} & \textrm{if } i = j, \\
	0 & \textrm{otherwise.}
	\end{cases}
	\label{eq:sbd}
\end{align}
This allows for a simple definition of the manifold via an equality constraint as
\begin{align}
	\Stdpm & = \left\{ Y \in \Rnp : \sbdsmall{YY\transpose\,} = I_n \right\}.
    \label{eq:Stdpmsbd}
\end{align}
The set is non-empty if $p \geq d$. It is connected if $p > d$. Counting dimensions yields $\dim \Stdpm = np - md(d+1)/2$.  The tangent space to $\Stdpm$ at $Y$ is a subspace of $\Rnp$ obtained by differentiating the equality constraint:
\begin{align}
	\T_Y\Stdpm & = \left\{ \dot Y \in \Rnp : \sbd{\dot Y Y\transpose + Y \dot Y\transpose\,}
	  = 0 \right\}.
	\label{eq:tangentspace}
\end{align}
Among the tangent vectors are all vectors of the form $Y\Omega$, for $\Omega \in \Rpp$ skew-symmetric: these correspond to ``vertical directions'', in the sense that following them does not affect the product $YY\transpose$ (at first order). Each tangent space is equipped with a restriction of the metric $\inner{\cdot}{\cdot}$, thus making $\Stdpm$ a Riemannian submanifold of $\Rnp$. The orthogonal projector from the embedding space $\Rnp$ to the tangent space at $Y$ is %
\begin{align}
	\Proj_Y(Z) & = Z - \sbdsmall{ZY\transpose\,} Y.
	\label{eq:Proj}
\end{align}
The total computational cost of a projection is thus $\mathcal{O}(m \cdot d^2 p) = \mathcal{O}(n d p)$ flops.

The optimization problem~\eqref{eq:RP} involves a function $g(Y) = f(YY\transpose)$ defined over $\Rnp$. Denote its classical, Euclidean gradient at $Y$ as $\nabla g(Y)$. The Riemannian gradient of $g$ at $Y$, $\grad\,g(Y)$, is defined as the unique tangent vector at $Y$ such that, for all tangent $\dot Y$, %
$
	\innersmall{\grad\,g(Y)}{\dot Y} = \innersmall{\nabla g(Y)}{\dot Y} . %
$
Naturally, this is given by the projection of the classical gradient to the tangent space~\citep[eq.\,(3.37)]{AMS08}:
\begin{align}
	\grad\,g(Y) & = \Proj_Y\left(\nabla g(Y)\right) = 2\,\Proj_Y \left(\nabla f(YY\transpose) Y \right),
	\label{eq:gradg}
\end{align}
where $\nabla f(X)$ is the classical gradient of $f$, a symmetric\footnote{$\nabla f(X)$ is symmetric because $f$ is formally defined over the symmetric matrices. If the gradient of $f$ over the square matrices is not symmetric, $\nabla f(X)$ is obtained by extracting its symmetric part.} matrix of size $n$, and we used~\eqref{eq:nablagnablaf}. Furthermore, denote by $\nabla^2 g(Y)$ the classical Hessian of $g$ at $Y$. This is a symmetric operator on $\Rnp$. The Riemannian Hessian of $g$ at $Y$ is a symmetric operator on the tangent space at $Y$ obtained as the projection of the derivative of the gradient vector field~\citep[eq.\,(5.15)]{AMS08}:
\begin{align}
	\Hess\,g(Y)[\dot Y] & = \Proj_Y\!\left( \D\big( Y \mapsto \Proj_Y\left(\nabla g(Y)\right) \big)(Y)[\dot Y] \right) \nonumber\\
	& = \Proj_Y\!\left( \nabla^2 g(Y)[\dot Y] - \sbd{\nabla g(Y) Y\transpose\,}\dot Y \right),
	\label{eq:Hessg}
\end{align}
where $\D$ denotes a classical directional derivative and we used $\Proj_Y \circ \Proj_Y = \Proj_Y$. For future reference, we note these expressions of the derivatives of $g$ in terms of those of $f$:
\begin{align}
	\nabla g(Y) & = 2\nabla f(X) Y, \ \textrm{ and} \label{eq:nablagnablaf} \\
	\nabla^2 g(Y)[\dot Y] & = 2\left( \nabla^2 f(X)[\dot X]Y + \nabla f(X)\dot Y \right), \textrm{ with } \dot X = \dot Y Y\transpose + Y\dot Y\transpose.
	\label{eq:nabla2gnabla2f}
\end{align}

Optimization algorithms on Riemannian manifolds typically are iterative. As such, they require a means of moving away from a point $Y$ along a prescribed tangent direction $\dot Y$, to reach a new point on the manifold: the next iterate. Since $Y+\dot Y$ does not, in general, %
belong to the manifold, extra operations are required. \emph{Retractions} achieve exactly this~\citep[\S\,4.1]{AMS08}. One possible retraction for~\eqref{eq:Stdpm} is as follows. For each $d\times p$ ``slice'' $i$ in $\{1, \ldots, m\}$,
\begin{align}
	\left(\Retr_Y(\dot Y)\right)_i & = U_i^{}V_i\transpose, \quad \textrm{ with } \qquad  Y_i + \dot Y_i = U_i^{}\Sigma_i^{} V_i\transpose,
	\label{eq:retraction}
\end{align}
where $U_i^{}\Sigma_i^{} V_i\transpose$ is a thin singular value decomposition of the \ith slice $Y_i + \dot Y_i$. This retraction projects each slice of $Y + \dot Y$ to the closest orthonormal matrix. Consequently, this is even a second-order retraction~\citep{absil2012retractions}. The total cost of computing a retraction is $\mathcal{O}(m\cdot (p^2 d + d^3))
= \mathcal{O}(np^2)$ flops.

\subsection{Convex geometry of the full search space}
\label{sec:geometryconvex}

The optimization problem~\eqref{eq:P} is defined over the compact convex set
\begin{align*}
	\calC & = \{ X \in \Snn : X \succeq 0 \textrm{ and } X_{ii} = I_d \textrm{ for } i = 1\ldots m\}.
\end{align*}
For $d = 1$, this is the \emph{elliptope}, or set of correlation matrices%
~\citep{laurent1996facial}. %
Often, we hope to recover matrices $X$ in $\calC$ such that $X$ has rank $d$, or such that off-diagonal blocks $X_{ij}$ are orthogonal. The following proposition shows that these two considerations are equivalent, and that the convex relaxation leading to $\calC$ is tight in that sense (the tightest relaxation would consider the convex hull of rank-$d$ matrices in $\calC$, but this is difficult to handle\footnote{Let $\tilde \calC$ be the convex hull of rank-$d$ matrices in $\calC$. The extreme points of $\tilde \calC$ are these matrices~\citep[Cor.\,18.3.1]{rockafellar1997convex}. Thus, optimizing a linear cost function over $\tilde \calC$ solves~($\textrm{RP}_d$), which is NP-hard. Hence, there probably does exist an efficient representation of $\tilde \calC$.}).
\begin{proposition}
	For all $X\in\calC$, all blocks $X_{ij}$ are in the convex hull of $\Od$, that is, $\sigma_{\mathrm{max}}(X_{ij}) \leq 1$. Furthermore, $\rank(X) = d$ if and only if $X_{ij} \in \Od$ for all $i, j$.
	\label{prop:Xijconvhull}
\end{proposition}
\begin{proof}
	Since $X$ is positive semidefinite, for all $i \neq j$, the submatrix formed by the blocks $X_{ii}, X_{ij}, X_{ji}$ and $X_{jj}$ is positive semidefinite. By Schur, this holds if and only if $X_{ij}\transpose X_{ij}^{} \preceq I_d$, which in turn happens if and only if all singular values of $X_{ij}$ are at most 1. The set of such matrices is the convex hull of all orthogonal matrices of size $d$~\citep{saunderson2014hulls}.
	
	Consider $Y \in \Stdpm$ such that $\rank(X) = p$ and $X = YY\transpose$. Clearly, if $p = d$, $X_{ij} = Y_i^{}Y_j\transpose$ is orthogonal, since all $Y_i$'s are orthogonal. Conversely, if $X_{ij}$ is orthogonal, then the rows of $Y_i$ and $Y_j$ span the same subspace. Indeed, $Y_i^{}Y_j\transpose Y_j^{}Y_i\transpose = I_d$. Multiply by $Y_i$ on the right. Now notice that $Y_k\transpose Y_k^{}$ is an orthogonal projector onto the subspace spanned by the rows of $Y_k$. Since $Y_i$ remains unaffected by such a projection first on the subspace of $Y_j$ then again on the subspace of $Y_i$, they must span the same subspace. Hence, fixing $j = 1$, for each $i$, there exists $Q_i \in \Od$ such that $Y_i = Q_i Y_1$. Finally, $Y = \diag(I_d, Q_2, \ldots, Q_m) (\mathds{1}_{m\times 1} \otimes Y_1)$ (where $\otimes$ is the Kronecker product), which confirms that $Y$ and $X$ have rank $d$. Notice that this proof further shows that $X$ has rank $d$ if and only if there is a spanning tree of edges $(i, j)$ on an $m$-nodes graph such that the $X_{ij}$'s are orthogonal (in which case they are all orthogonal).
\end{proof}
The set $\calC$ may be decomposed into \emph{faces} of various dimensions.
\begin{definition}[faces, \S18 in \citep{rockafellar1997convex}]
	A face of $\calC$ is a convex subset $\calF$ of $\calC$ such that every (closed) line segment in $\calC$ with a relative interior point in $\calF$ has both endpoints in $\calF$. The empty set and $\calC$ itself are faces of $\calC$. %
\end{definition}
By~\citep[Thm.\,18.2]{rockafellar1997convex}, the collection of relative interiors\footnote{The relative interior of a singleton is the singleton.} of the non-empty faces forms a partition of $\calC$. That is, each $X\in\calC$ is in the relative interior of exactly one face of $\calC$, called $\calF_X$ .
Furthermore, all faces of $\calC$ are exposed~\citep[Cor.\,1]{ramana1995geometric}, that is, for every face $\calF$, there exists a linear function $f$ such that $\calF$ is the set of solutions of~\eqref{eq:P}.
Of particular interest are the zero-dimensional faces of $\calC$ (singletons), called its \emph{extreme points}.
\begin{definition}[Extreme and exposed points]
	$X\in\calC$ is an extreme point of $\calC$ if there does not exist $X', X'' \in \calC\backslash\{X\}$ and $0 < \lambda < 1$ such that $X = \lambda X' + (1-\lambda) X''$. $X$ is an \emph{exposed point} of $\calC$ if there exists $C$ such that $X$ is the unique maximizer of $\inner{C}{X}$ in $\calC$.
\end{definition}
In other words, $X$ is extreme if it does not lie on an open line segment included in $\calC$. Since $\calC$ is compact, it is the convex hull of its extreme points~\citep[Cor.\,18.5.1]{rockafellar1997convex}.
Extreme points are of interest notably because they often arise as the solution of optimization problems. Specifically, if $f$ is a concave function (in particular, if $f$ is linear), then $f$ attains its minimum on $\calC$ at one of its extreme points~\citep[Cor.\,32.3.2]{rockafellar1997convex}.

Following the construction in the proof of~\citep[Thm.\,2.1]{pataki1998rank}, given $Y\in\Rnp$ of full rank
such that $X = YY\transpose$ ($\rank(X) = p$), we find that
\begin{align}
\calF_X & = \left\{  \tilde X %
= Y(I_p + A)Y\transpose : A \in \ker \calL_X \textrm{ and } I_p + A \succeq 0\right\}, \textrm{ with}
\label{eq:calF} \\
\calL_X & \colon \Spp \to (\Sdd)^m  \colon A \mapsto \calL_X(A) = \left( Y_1^{}AY_1\transpose, \cdots, Y_m^{}AY_m\transpose \right).
\label{eq:calL}
\end{align}
The dimension of $\calF_X$ is the dimension of the kernel of $\calL_X$. The rank-nullity theorem gives a lowerbound (see also Theorem~\ref{thm:dimFbounds} for an upperbound):
\begin{align}
	\dim \calF_X = \frac{p(p+1)}{2} - \rank \calL_X \geq \frac{p(p+1)}{2} - m \frac{d(d+1)}{2} \triangleq \Delta.
	\label{eq:dimF}
\end{align}
It follows that extreme points $X$ (i.e., points such that $\dim \calF_X = 0$) have small rank:
\begin{align}
	d \ \leq \ \rank(X) \ \leq \ p^* := \left(\sqrt{1+4md(d+1)}-1\right)/2.
	\label{eq:patakirank}
\end{align}
Note that $\Delta \geq 0$ when $p \geq p^*$.
\begin{remark}
	For linear $f$,~\eqref{eq:P} admits an extreme point as global optimizer, so that~\eqref{eq:RP} and~\eqref{eq:P} have the same optimal value as soon as $p \geq p^*$~\eqref{eq:patakirank}. In other words: for linear $f$,~\eqref{eq:RP} is not NP-hard if $p \geq p^*$.
	\label{rem:linearlargepiseasy}
\end{remark}
Not all feasible $X$'s with rank as in~\eqref{eq:patakirank} are extreme. For example, setting $d = 1$ and $m \geq 3$ as in Figure~\ref{fig:geometry}, select two distinct, admissible matrices of rank 1, $X_0$ and $X_1$. For all $0 < \lambda < 1$, the matrix $X_\lambda = \lambda X_1 + (1-\lambda)X_0$, lying on the open line segment between $X_0$ and $X_1$, is admissible and has rank 2. Thus, $X_\lambda$ satisfies~\eqref{eq:patakirank}, but it is not an extreme point, by construction. Notwithstanding, the expectation that $\calL_X$ is generically of full rank suggests that almost all feasible $X$'s satisfying~\eqref{eq:patakirank} should be extreme; an intuition that is supported by Figure~\ref{fig:geometry}. More generally, in Theorem~\ref{thm:genericfacedimension}, we prove for $d=1$ that $\dim \calF_X = \Delta$ for almost all $X$ of rank $p$. %

Many applications look for solutions of rank $d$. All $X$'s of rank $d$ are exposed (hence extreme), meaning they can all be recovered as unique solutions of~\eqref{eq:P}.
\begin{proposition}
	For all $X\in\calC$, $\sqfrobnormsmall{X} \leq m^2d$. Furthermore, $\sqfrobnormsmall{X} = m^2d$ if and only if $\rank(X) = d$. In particular, each $X\in\calC$ of rank $d$ is an exposed extreme point of $\calC$.
	\label{prop:rankdmaxnorm}
\end{proposition}
\begin{proof}
	Let $\sigma_1(X_{ij}) \geq \cdots \geq \sigma_d(X_{ij})$ denote the singular values of $X_{ij}$. By Proposition~\ref{prop:Xijconvhull}, $\sigma_k(X_{ij}) \leq 1$ for all $i,j,k$. Hence,
	\begin{align}
		\sqfrobnormsmall{X} & = \sum_{i,j = 1}^m \sqfrobnormsmall{X_{ij}} = \sum_{i,j = 1}^m \sum_{k=1}^{d} \sigma_k^2(X_{ij}) \leq m^2d.
	\end{align}
	The upperbound is attained if and only if $\sigma_k(X_{ij}) = 1$ for all $i,j,k$, thus, if and only all $X_{ij}$'s are orthogonal. By Proposition~\ref{prop:Xijconvhull}, this is the case if and only if $\rank(X) = d$. Now consider $X$ has rank $d$. We show it is exposed (and hence extreme):
	\begin{align}
		\max_{\hat X \in \calC} \ \innersmall{X}{\hat X} & \leq \max_{\sqfrobnormsmall{\hat X} \leq m^2d} \ \innersmall{X}{\hat X} \leq \frobnormsmall{X} \cdot \max_{\sqfrobnormsmall{\hat X} \leq m^2d} \ \frobnormsmall{\hat X} = m^2d.
	\end{align}
	The second inequality follows by Cauchy-Schwarz, and equality is attained if and only if $\hat X = X$, which is in $\calC$. Thus, both max problems admit $X$ as unique solution, confirming that $X$ is an exposed extreme point.
\end{proof}
\begin{remark}
	Proposition~\ref{prop:rankdmaxnorm} is an extension of~\citep[Thm.\,1]{malick2007spherical} to the case $d > 1$. We note that the proof in that reference does not generalize to $d > 1$. %
	\label{rem:normsandrank}
\end{remark}

\section{From second-order critical points $Y$ to KKT points $X$} \label{sec:analysis}

In this section, we show that rank-deficient second-order critical points $Y$ of~\eqref{eq:RP} yield KKT points $X = YY\transpose$ of ~\eqref{eq:P}. Furthermore, when $Y$ is second-order critical but $X$ is not KKT,
it is shown how to escape the saddle point by increasing $p$. If $p$ increases all the way to $n$, then all second-order critical points reveal KKT points. The proofs parallel those in~\citep{journee2010low}. %
The main novelty is explicit bounds on $p$ such that \emph{all} second-order critical points of~\eqref{eq:RP} reveal KKT points. The proofs bring us to consider the facial structure of $\calC$.

A key ingredient for all proofs in this section is the availability of an explicit matrix $S(X)$~\eqref{eq:S} which is positive semidefinite if and only if $X$ is KKT (Theorem~\ref{thm:KKTS}). The formula for $S$ is simply read off from the first-order optimality conditions of~\eqref{eq:RP}, owing to smoothness of the latter.

\begin{lemma}[Necessary optimality conditions for~\eqref{eq:P}]\label{lem:kktpoint}
$X\in\calC$ is called a \emph{KKT point} for~\eqref{eq:P} if there exist a symmetric matrix $\hat S\in\Snn$ and a symmetric, block-diagonal matrix $\hat \Lambda \in \Snn$ (dual variables) such that
\begin{align*}
\hat S X & = 0, & \hat S & = \nabla f(X) + \hat \Lambda, & \textrm{ and}&  & \hat S & \succeq 0.
\end{align*}
If $X$ is a local optimizer for~\eqref{eq:P}, then $X$ is a KKT point. If $f$ is convex, all KKT points are global optimizers. %
\end{lemma}
\begin{proof}
Apply Theorems 3.25 and 3.34, and Example 3.36 in~\citep{ruszczynski2006nonlinear}. KKT conditions are necessary since Slater's condition holds: $I_n$ is feasible for~\eqref{eq:P} and it is strictly positive definite.
\end{proof}

\begin{lemma}[Necessary optimality conditions for~\eqref{eq:RP}]\label{lem:localopt}
Let $Y\in\Stdpm$ and $X = YY\transpose$. A \emph{critical point} of~\eqref{eq:RP} satisfies $\grad\,g(Y) = 0$, that is,
\begin{align}
	\Big(\nabla f(X) - \sbd{\nabla f(X) X} \Big)Y = 0.
	\label{eq:firstordercondition}
\end{align}
A \emph{second-order critical point} is a critical point which satisfies $\Hess\,g(Y) \succeq 0$, that is,
\begin{align}
	\textrm{ for all } \dot Y \in \T_Y\Stdpm, \quad \inner{\dot Y}{\nabla^2 g(Y)[\dot Y] - \sbd{\nabla g(Y) Y\transpose\,}\dot Y} \geq 0.
	\label{eq:secondordercondition}
\end{align}
If $Y$ is a local optimizer for~\eqref{eq:RP}, then it is a second-order critical point.
\end{lemma}
\begin{proof}
This is a direct generalization of the classical necessary optimality conditions for unconstrained optimization~\citep[Prop.\,1.1.1]{bertsekas1995nonlinear}
to the Riemannian setting, as per the formalism in~\citep{AMS08,yang2012optimality}.
Use equations~\eqref{eq:Proj}, \eqref{eq:gradg} and~\eqref{eq:Hessg} and the fact that $\Proj_Y$ is self-adjoint to obtain equations~\eqref{eq:firstordercondition} and~\eqref{eq:secondordercondition}. %
\end{proof}
Lemmas~\ref{lem:kktpoint} and~\ref{lem:localopt} suggest the definition of an (as yet merely tentative) formula for the dual certificate $\hat S$, based on~\eqref{eq:firstordercondition}:
\begin{align}
	S & = S(X) = \nabla f(X) - \sbd{\nabla f(X) X}.
	\label{eq:S}
\end{align}
Indeed,
for any critical point $Y$ of~\eqref{eq:RP}, it holds (with $X = YY\transpose$) that $SY = 0$, so that $SX = 0$.
In that case, for $p=d$, $S$ can be advantageously interpreted as a graph Laplacian (up to a change of variable\footnote{For $Y$ a critical point of~\eqref{eq:RP} with $p=d$, we have $Y = \left(\begin{smallmatrix} Y_1\transpose & \cdots & Y_m\transpose \end{smallmatrix}\right)\transpose$ with orthogonal $Y_i$'s, and $SY = 0$. Write $\nabla f(YY\transpose) := -C$ for short. $SY = 0$ implies, after some algebra, that the $(CYY\transpose)_{ii}$'s are symmetric. This can be used to see that $\tilde S = \diag(Y_1, \ldots, Y_m)\,\transpose S(YY\transpose)\, \diag(Y_1, \ldots, Y_m)$ is a matrix with off-diagonal blocks equal to $-Y_i\transpose C_{ij}^{} Y_j^{}$ and diagonal blocks equal to $\sum_{j \neq i} Y_i\transpose C_{ij}^{} Y_j^{}$. Thus, $\tilde S$ is exactly the Laplacian of the graph with $m$ nodes and edge ``weights'' given by the matrices $Y_i\transpose C_{ij}^{} Y_j^{}$.}), as is often the case for dual certificates of estimation problems on graphs~\citep{bandeira2014tightness,demanet2013convex}. We now show that $S(X)$ is indeed the unique possible dual certificate for any feasible $X$.
(This is similar to, but different from, Theorem 4 in~\citep{journee2010low}; a complex version appears in~\citep{bandeira2014tightness} for $d=1$.)
\begin{theorem}[$S$ is the right certificate]
\label{thm:KKTS}
$X \in \calC$ is a KKT point for~\eqref{eq:P} if and only if $S$~\eqref{eq:S} is positive semidefinite. If so, $\hat S = S$ is the unique dual certificate for Lemma~\ref{lem:kktpoint}.
\end{theorem}
\begin{proof}
We show the if and only if parts of the first statement separately.
\begin{itemize}
\item[$\mathbf{\Leftarrow}$:] By construction, $\trace(SX) = 0$. Since both $S$ and $X$ are positive semidefinite, this implies $SX = 0$. Apply Lemma~\ref{lem:kktpoint} with $\hat S = S$ and $\hat \Lambda = - \sbd{\nabla f(X) X}$. %
\item[$\mathbf{\Rightarrow}$:] Since $X$ is a KKT point for~\eqref{eq:P}, there exist $\hat S \succeq 0$ and $\hat \Lambda$ symmetric, block-diagonal satisfying the conditions in Lemma~\ref{lem:kktpoint}. In particular, $\hat S X = 0$ and $\nabla f(X) = \hat S - \hat \Lambda$. Thus, $\nabla f(X) X = -\hat \Lambda X$ and $\sbd{\nabla f(X) X} = -\sbd{\hat \Lambda X} = -\hat \Lambda$. Here, we used both the fact that $\hat \Lambda$ is symmetric, block-diagonal and the fact that
$X_{ii} = I_d$.
Consequently, $S = \nabla f(X) - \sbd{\nabla f(X) X} = \hat S - \hat \Lambda + \hat \Lambda = \hat S \succeq 0$.
\end{itemize}
The last point also shows there exists only one pair $(\hat S, \hat \Lambda)$ certifying $X$ is a KKT point.
\end{proof}
Notice how the Riemannian structure underlying problem~\eqref{eq:P} made it possible to simply read off an analytical expression for a dual certificate from the necessary optimality conditions of~\eqref{eq:RP}. This smooth geometry also leads to uniqueness of the dual certificate (this is connected to nondegeneracy~\citep[Thm.\,7]{alizadeh1997complementarity}). Theorem~\ref{thm:KKTS} makes for an unusually comfortable situation
and will be helpful throughout the paper.
%

%
%
%
%
%
%
%
%
%
%
For convex $f$, we can make the following statement regarding uniqueness of the solution.
\begin{theorem}
Assume $f$ is convex. If $X \in \calC$ is an extreme point for~\eqref{eq:P} (which is true in particular if $\rank(X) = d$), and $S \succeq 0$, and $\rank(X) + \rank(S) = n$ (strict complementarity),
then $X$ is the unique global optimizer of~\eqref{eq:P}.
\label{thm:fconvexglobaloptunique}
\end{theorem}
\begin{proof}

From Theorem~\ref{thm:KKTS}, it is clear that $X$ is a global optimizer. We prove by contradiction that it is unique. Let $X' \neq X$ be another global optimizer. Since \eqref{eq:P} is a convex problem in this setting, $f$ is constant over the whole (optimal) segment $t \mapsto X + t(X'-X)$ for $t \in [0, 1]$. Hence, the directional derivative of $f$ at $X$ along $\dot X = X'-X$ is zero: 
\begin{align*}
 	0 	& = \innersmall{\nabla f(X)}{\dot X} \\
 		& = \innersmall{S + \sbd{\nabla f(X) X}}{\dot X} & & (\textrm{definition of $S$~\eqref{eq:S}})\\
 		& = \innersmall{S}{\dot X} & & (\textrm{diagonal blocks of $\dot X$ are zero})\\
 		& = \innersmall{S}{X'} & & (\textrm{$SX = 0$}).
\end{align*}
Since both $S$ and $X'$ are positive semidefinite, it ensues that $SX' = 0$. (Note that for linear $f$, this shows $S$ is the dual certificate for all global optimizers of~\eqref{eq:P}, not only for $X$.) Hence, $S \dot X = 0$.

Let $p = \rank(X)$ and $Y \in \Rnp$ be a full-rank matrix such that $X = YY\transpose$. Strict complementarity and $SX = 0$ imply the columns of $Y$ form a basis for the kernel of $S$. Hence, $S\dot X = 0$ ensures that $\dot X = YAY\transpose$ for some $A \in \Spp$, $A \neq 0$, such that $\sbdsmall{YAY\transpose} = 0$. In other words, $X' = X + \dot X = Y(I_p + A)Y\transpose$ is in the face $\calF_X$. This is a contradiction, since $\calF_X = \{X\}$.
%
%
\end{proof}
\begin{remark}
	In general, the condition that $X$ be an extreme point in the previous theorem cannot be removed. Indeed, if $f(X) \equiv 0$, then all admissible $X$'s are globally optimal and $S(X) \equiv 0 \succeq 0$. In particular, $X = I_n$ satisfies strict complementarity, but if $m > 1$, it is not extreme, and it is not a unique global optimizer. Likewise, any rank-$d$ admissible $X$ is extreme and globally optimal, but does not satisfy strict complementarity. Similar examples can be built with nonzero linear costs $f(X) = \inner{C}{X}$, where the sparsity pattern of $C$ corresponds to a disconnected graph.
	
	Conversely, it is not true in general that uniqueness of the global optimizer
	implies extremeness or strict complementarity. Simply consider $f(X) = \sqfrobnormsmall{X-X_0}$ with $X_0\in\calC$: the global optimizer $X = X_0$ is unique, and $\nabla f(X_0) = 0$, so that $S(X_0) = 0$. For $m$ large enough, $X_0$ can be chosen to be both not extreme and rank deficient. For an illustration of this in semidefinite programs, see the nice example after Theorem 10 in~\citep{alizadeh1997complementarity}.
\end{remark}

Continuing with convex cost functions, KKT points of~\eqref{eq:P} coincide with global optimizers. This and the fact that~\eqref{eq:RP} is a relaxation of~\eqref{eq:P} lead to the following summary regarding global optimality conditions. For~\eqref{eq:RP}, these sufficient conditions are conclusive whenever the relaxation is tight.
\begin{corollary}[Global optimality conditions]
Assume $f$ is convex and let $Y\in\Stdpm$; $X = YY\transpose$ is globally optimal for~\eqref{eq:P} if and only if $S\succeq 0$. If so, then $Y$ is a global optimizer for the nonconvex problem~\eqref{eq:RP}.
Furthermore, if $X$ is extreme (in particular, if $\rank(X) = d$) and if $\rank(X) + \rank(S) = n$, then $X$ is the unique global optimizer of~\eqref{eq:P} and $Y$ is the unique global optimizer of~\eqref{eq:RP}, up to orthogonal action $YQ$, $Q \in \Op$.
\end{corollary}

Returning to the general case of $f$ not necessarily convex, we establish links between second-order critical points of~\eqref{eq:RP} and KKT points of~\eqref{eq:P}.
In doing so, it is useful to reformulate the second-order optimality condition~\eqref{eq:secondordercondition} on~\eqref{eq:RP} in terms of $S$~\eqref{eq:S} and $f$. Let $Y\in\Stdpm$ and $X = YY\transpose$. Then, $Y$ is a second-order critical point for~\eqref{eq:RP} if and only if it is critical and, for all $\dot Y \in \T_Y\Stdpm$, with $\dot X = \dot Y Y\transpose + Y\dot Y\transpose$, it holds that (using~\eqref{eq:nabla2gnabla2f}):
\begin{align}
	\inner{\dot Y}{\Hess\,g(Y)[\dot Y]} & = \inner{\dot Y}{\nabla^2 g(Y)[\dot Y] - \sbd{\nabla g(Y)Y\transpose} \dot Y} \nonumber\\
	& = 2\innerbig{\dot Y}{\nabla^2 f(X)[\dot X]Y + \nabla f(X)\dot Y - \sbd{\nabla f(X)X}\dot Y} \nonumber\\
	& = 2\innerbig{\dot Y}{\nabla^2 f(X)[\dot X]Y + S\dot Y} \nonumber\\
	& = \innerbig{\dot X}{\nabla^2 f(X)[\dot X]} + 2\innerbig{\dot Y}{S \dot Y} \geq 0.
	\label{eq:secondorderconditionbis}
\end{align}
Since~\eqref{eq:RP} is essentially equivalent to~\eqref{eq:P} with the additional constraint $\rank(X) \leq p$, assuming $Y$ is optimal for~\eqref{eq:RP}, we expect $X$ to be a KKT point at least if either of the following holds: (1) if $Y$ is rank deficient, since then the extra constraint is not active, meaning it is ``as if'' we were solving~\eqref{eq:P}; or (2) if $p = n$, since then the extra constraint is vacuous. The two following theorems show this still holds for second-order critical points $Y$.
\begin{theorem} %
If $Y$ is a rank-deficient, second-order critical point for~\eqref{eq:RP}, then $X = YY\transpose$ is a KKT point for~\eqref{eq:P}.%
\label{thm:rankdeficientY}
\end{theorem}
\begin{proof}
By Theorem~\ref{thm:KKTS}, we must show that $S$~\eqref{eq:S} is positive semidefinite.
Since $Y$ is rank deficient, there exists $z \in \Rp$ such that $z \neq 0$ and $Yz = 0$. Furthermore, for all $x\in\Rn$, the matrix $\dot Y = xz\transpose$ is such that $Y\dot Y\transpose = 0$. In particular, $\dot Y$ is a tangent vector at $Y$~\eqref{eq:tangentspace}. Since $Y$ is second-order critical, inequality~\eqref{eq:secondorderconditionbis} holds, and here simplifies to:
%
%
%
%
%
%
%
\begin{align*}
	\innerbig{\dot Y}{S \dot Y} = \innerbig{xz\transpose}{Sxz\transpose\,} = \|z\|^2 \cdot x\transpose S x \geq 0.
\end{align*}
This holds for all $x\in\Rn$. 
Thus, $S$ is positive semidefinite.
\end{proof}
\begin{theorem} %
If $Y$ is square ($p=n$) and it is a second-order critical point for~$(\textrm{\emph{RP}}_n)$, then $X = YY\transpose$ is a KKT point for~\eqref{eq:P}. If $Y$ is full-rank, it needs only be first-order critical for $X$ to be a KKT point.
\label{thm:squareY}
\end{theorem}
\begin{proof}
If $Y$ is rank deficient, then the result follows from Theorem~\ref{thm:rankdeficientY}. If $Y$ is full rank, then it is invertible. Since $Y$ is also a critical point, first-order optimality conditions~\eqref{eq:firstordercondition} imply $SY = 0$, hence $S = 0$.
This completes the proof, as per Theorem~\ref{thm:KKTS}.
\end{proof}

In particular, if $f$ is convex and~\eqref{eq:P} has a unique solution of rank $r$, then all second-order critical points of~\eqref{eq:RP} have rank either $r$ or $p$. %
Thus, if $p$ is larger than the rank of a solution of~\eqref{eq:P}, we may hope that minimizing~\eqref{eq:RP} until we reach a second-order critical point will result in a rank-deficient $Y$, revealing a KKT point. Unfortunately, in general, we cannot guarantee rank deficiency beforehand. For those cases, the following theorem and corollary provide a means of escaping unsatisfactory
critical points.

\begin{theorem}[Escape direction (rank-deficient)]\label{thm:escaperankdeficient}
	Let $Y$ be a rank-deficient critical point for~\eqref{eq:RP} such that $X = YY\transpose$ is not a KKT point of~\eqref{eq:P}. Then, for all nonzero vectors $z\in\Rp$ and $u\in\Rn$ such that $Yz = 0$ and $u\transpose S u < 0$~\eqref{eq:S}, $\dot Y = u z\transpose \in \T_Y\Stdpm$ is a descent direction for $g$ from $Y$.
\end{theorem}
\begin{proof}
By Theorem~\ref{thm:KKTS}, $u$ exists because $X$ is not a KKT point. Since $Y \dot Y\transpose = 0$, $\dot Y$ is indeed tangent at $Y$~\eqref{eq:tangentspace}. From~\eqref{eq:secondorderconditionbis}, it follows that
\begin{align*}
\innerbig{\dot Y}{\Hess\,g(Y)[\dot Y]} = 2\innerbig{\dot Y}{S \dot Y} = 2\|z\|^2 \cdot u\transpose S u < 0.
\end{align*}
As a result, this truncated Taylor expansion holds (we use the fact that the retraction~\eqref{eq:retraction} is second-order and assume $\|z\| = 1$):%
\footnote{
	We remark that the residue is of order 4 rather than 3, since $\phi$ is an even function of $t$. Indeed, let $J = \diag(1, \ldots, 1, -1)\in\Op$ and consider the SVD of $Y = U\Sigma V\transpose$, such that the last column of $V$ is $z/\|z\|$. Then, $Y - t\dot Y = (YV - t\dot YV)V\transpose = (Y + t\dot Y)VJV\transpose$, because the last column of $YV$ is zero, and $\dot YV$'s only nonzero column is the last one. Hence, $\Retr_{Y}(-t\dot Y) = \Retr_{Y}(t\dot Y)VJV\transpose$. These yield the same matrix $X$, thus the same value of $g$.  If the residue is bounded by $Lt^4$, then $t = \sqrt{-u\transpose Su / 2L}$ ensures $\phi(t) \leq \phi(0) - (u\transpose Su)^2/4L$. In particular, for $f(X) = \inner{C}{X}$, the fourth-order term follows and could be used to estimate $L$, to be used in a line-search for the escape:
	\begin{align*}
	&\frac{1}{4}t^4 \cdot \left[ \innersmall{C}{AXA} + u\transpose D\big( 3\sbd{CX} - 4C \big)u \right] + \mathcal{O}(t^6), \\
	&\textrm{with } A  = \sbdsmall{uu\transpose} \textrm{ and } D = \diag(\|u_1\|^2, \ldots, \|u_m\|^2) \otimes I_d.
	\end{align*}
}
\begin{align*}
\phi(t) & := g(\Retr_{Y}(t\dot Y)) = g(Y) + (u\transpose S u) t^2 + \mathcal{O}(t^4).
\end{align*}
Since $u\transpose S u < 0$,
there exists $t_0 > 0$ such that $\phi(t) < \phi(0)$ for all $t$ in $]0, t_0[$.\qedhere
\end{proof}
\begin{corollary}[Escape direction (full-rank)]\label{cor:escapefullrank}
Let $Y$ be a full-rank critical point for~\eqref{eq:RP} such that $X = YY\transpose$ is not a KKT point of~\eqref{eq:P}. Let $Y_+ = \begin{pmatrix} Y & 0_{n\times (p_+ - p)}\end{pmatrix} \in \St(d, p_+)^m$. Then, (a) $g(Y_+) = g(Y)$; (b) $Y_+$ is a critical point for~$(\text{\emph{RP}}_{p_+})$; and (c) for all $u\in\Rn$ such that $u\transpose S u < 0$, $\dot Y = u e_{p+1}\transpose \in \T_{Y_+}\St(d, p_+)^m$ is a descent direction for $g$ from $Y_+$, where $e_{p+1} \in \reals^{p_+}$ is a zero vector except for its $(p+1)^\textrm{st}$ entry, equal to 1.
\end{corollary}
\begin{proof}
Since $Y_+^{} Y_+\transpose = YY\transpose$, $Y_+$ is indeed feasible for~$(\text{RP}_{p_+})$ and $g(Y_+) = g(Y)$. Since $Y$ is critical, $SY = 0$, so $SY_+ = 0$:
$Y_+$ is a critical point. %
The rest follows from Theorem~\ref{thm:escaperankdeficient}.
\end{proof}
Later in this section, we show how to escape full-rank points without increasing the rank, under additional assumptions---see Proposition~\ref{prop:facerankreduction}.

An important question remains: for moderate $p$, should we expect to encounter second-order critical points $Y$ that do not correspond to KKT points of~\eqref{eq:P}? We provide partial answers for concave $f$ below. In particular, this covers the important case of linear costs. The stronger results do not include strictly convex functions, as for these \eqref{eq:P} can have solutions of arbitrary rank.

The result below shows that $Y$ can only be a critical point of~\eqref{eq:RP} if $YY\transpose$ is a critical point for~\eqref{eq:P} restricted to the face $\calF_X$; and similarly for second-order critical points. This brings a useful corollary.
\begin{lemma}
	Let $Y \in \Stdpm$ and $X = YY\transpose \in \calC$. 
	If $Y$ is a critical point (resp., a second-order critical point) of~\eqref{eq:RP}, then $X$ is a critical point (resp., a second-order critical point) of $\min_{\tilde X\in \calF_X} f(\tilde X)$, where $\calF_X$~\eqref{eq:calF} is the face of $\calC$ which contains $X$ in its relative interior.
	\label{lem:faces}
\end{lemma}
\begin{proof}
	If $Y$ is first-order critical, then $SY = 0$~\eqref{eq:S}. This implies that $\nabla f(X)$ is orthogonal to all directions $\dot X = Y\dot Y\transpose + \dot Y Y\transpose$ for $\dot Y \in \T_Y\Stdpm$. Indeed, $\sbdop(\dot X) = 0$, hence
	\begin{align*}
		\innersmall{\nabla f(X)}{\dot X} = \innersmall{S}{\dot X} = 0.
	\end{align*}
	(The subspace spanned by all such $\dot X$'s has dimension $np - m\frac{d(d+1)}{2} - \frac{p(p-1)}{2}$.)
	From~\eqref{eq:calF}, observe that $\dot X$ is parallel to the face $\calF_X$ iff $\dot X = YAY\transpose$ for $A\in\Spp$ with $\sbd{YAY\transpose} = 0$. Thus, $\nabla f(X)$ is, in particular, orthogonal to $\calF_X$.
	This shows $X$ is a first-order critical point for $\min_{\tilde X\in \calF_X} f(\tilde X)$, since $X$ is in the relative interior of the face.
	Further consider~\eqref{eq:secondorderconditionbis} for second-order critical $Y$ and all $\dot X = YAY\transpose$. Since $\dot Y = \frac{1}{2}YA$ and $SY = 0$, it follows that
	\begin{align*}
	0 \leq \innerbig{\dot X}{\nabla^2 f(X)[\dot X]} + 2\innerbig{\dot Y}{S \dot Y} = \innerbig{\dot X}{\nabla^2 f(X)[\dot X]}.
	\end{align*}
	Thus, $X$ is a second-order critical point for the face-restricted optimization problem.
\end{proof}
%
%
%
%
%
%
%
%
%
%
%
%
%
%
The following corollary can be put in perspective with~\citep[Thm.\,3.4]{burer2005local}. The latter states a similar result for~\eqref{eq:P} with general linear equality constraints, for linear $f$. Their result characterizes local optimizers of~\eqref{eq:RP}, whereas the following result characterizes first- and second-order critical points (computationally more manageable objects).
\begin{corollary} Theorem~\ref{thm:rankdeficientY} and Lemma~\ref{lem:faces} imply the following, for $Y\in\Stdpm$ and $X = YY\transpose\in\calC$.
	\begin{itemize}
		\item If $f$ is linear and $Y$ is critical, then $f$ is constant over $\calF_X$. If furthermore $Y$ is second-order critical and $p > \floorr{p^*}$~\eqref{eq:patakirank}, then either $X$ is globally optimal for~\eqref{eq:P}, or ($Y$ has full rank and) the face $\calF_X$ has positive dimension~\eqref{eq:dimF} and is suboptimal.
		\item If $f$ is strongly concave and $Y$ is second-order critical, then $X$ is an extreme point (since $\dot X = YAY\transpose \neq 0 \Rightarrow \innersmall{\dot Y}{\Hess\,g(Y)[\dot Y]} = \innersmall{\dot X}{\nabla^2 f(X)[\dot X]} < 0$, a contradiction). In particular, all second-order critical points of~\eqref{eq:RP} have rank at most $\floorr{p^*}$, and $p > \floorr{p^*} \Rightarrow X$ is KKT (since $Y$ must be rank-deficient).
		\item If $f$ is convex (resp., strictly convex) and $Y$ is critical, then $X$ is optimal (resp., the unique optimizer) for~\eqref{eq:P} restricted to $\calF_X$.
	\end{itemize}
	\label{cor:concaveffacerestriction}
\end{corollary}
For linear $f$, Corollary~\ref{cor:concaveffacerestriction} is not quite sufficient to determine how large $p$ must be to exclude ``bad'' second-order critical points. Paraphrasing the comment following~\citep[Thm.\,3.4]{burer2005local}, the latter showed that, for linear $f$ and $p > p^*$ (essentially), local optima of~\eqref{eq:RP} are global optima, \emph{with the caveat that positive-dimensional faces (over which $f$ must be constant) may harbor non-global local optima}. In the literature, this has sometimes been quoted as saying that local optima are global optima if $f$ is not constant over any proper face of $\calC$~\citep[see, e.g.,][footnote 3]{mccoy2011robustpca}, but there is no indication that this is a mild condition.\footnote{In fact, we found in numerical experiments (not reported) that, for $d = 1$, $\Delta = 1$ (thus, almost all faces have dimension 1 and $p > p^*$) and a random linear cost $\innersmall{C}{X}$, we could easily find a face $\calF_X$ of dimension 1 over which the cost is constant but not optimal.}

We thus set out to further refine the implications of second-order criticality of $Y$. We do so by leveraging the tight relationship~\eqref{eq:secondorderconditionbis} between the Hessian of the cost on~\eqref{eq:RP} and the dual certificate $S$.
\begin{theorem}
	Assume $f$ is concave. Let $Y\in\Stdpm$ be a second-order critical point for~\eqref{eq:RP}. The matrix $X = YY\transpose$ belongs to the relative interior of the face $\calF_X$~\eqref{eq:calF}. If $Y$ is rank-deficient, then $S\succeq 0$. If $Y$ is full-rank, then $S$ has at most
	\begin{align}
		\floorr*{\frac{\dim \calF_X - \Delta}{p}}
	\end{align}
	negative eigenvalues.
	Hence, if $X$ is not a KKT point for~\eqref{eq:P}, then it has rank $p$ and $\dim \calF_X \geq \Delta + p$.
	\label{thm:eigS}
\end{theorem}
\begin{proof}
	Given Theorems~\ref{thm:KKTS} and~\ref{thm:rankdeficientY}, we need only focus on full-rank $Y$'s.
	Since $Y$ is second-order critical, inequality~\eqref{eq:secondorderconditionbis} holds. Since furthermore $f$ is concave, $\innersmall{\dot X}{\nabla^2 f(X)[\dot X]} \leq 0$ for all $\dot X$, so that $\innersmall{\dot Y}{S\dot Y} \geq 0$ for all $\dot Y \in \T_Y\Stdpm$. Let $k = \dim \Stdpm$ and $U\in\reals^{np\times k}$, $U\transpose U = I_k$ denote an orthonormal basis of the space spanned by the vectorized tangent vectors $\vecc(\dot Y)$. Then, $U\transpose(I_p \otimes S)U$ is positive semidefinite (for linear $f$, it inherits
	the spectrum of $\frac{1}{2} \Hess\,g(Y)$). On the other hand, since $Y$ is a critical point, $SY = 0$. Let $V\in\reals^{np\times p^2}, V\transpose V = I_{p^2}$ denote an orthonormal basis of the space spanned by the vectors $\vecc(YR)$ for $R\in\Rpp$. Clearly, $(I_p\otimes S)V = 0$. Let $k'$ denote the dimension of the space spanned by the columns of both $U$ and $V$, and let $W\in\reals^{np\times k'}, W\transpose W = I_{k'}$ be an orthonormal basis for this space. It follows that $M = W\transpose (I_p\otimes S)W$ is positive semidefinite.
	
	Let $\lambda_0 \leq \cdots \leq \lambda_{n-1}$ denote the eigenvalues of $S$. Likewise, $\tilde \lambda_0 \leq \cdots \leq \tilde \lambda_{np-1}$ denote the eigenvalues of $I_p\otimes S$. These are simply the eigenvalues of $S$, each repeated $p$ times, thus: $\tilde \lambda_i = \lambda_{\floorr{i/p}}$. Let $\mu_0 \leq \cdots \leq \mu_{k'-1}$ denote the eigenvalues of $M$. The Cauchy interlacing theorem states that, for all $i$,
	\begin{align}
		\tilde \lambda_i \leq \mu_i \leq \tilde \lambda_{i + np - k'}.
		\label{eq:cauchyinterlace}
	\end{align}
	In particular, since $M\succeq 0$, we have $0 \leq \mu_0 \leq \lambda_{\floorr{(np - k')/p}}$. It remains to determine $k'$.
	
	From Section~\ref{sec:geometryriemannian}, recall that $k = np - md(d+1)/2$. We now investigate how many new dimensions $V$ adds to $U$. All matrices $R\in\Rpp$ admit a unique decomposition as $R = R_\mathrm{skew} + R_{\ker \calL} + R_{(\ker \calL)^\bot}$, where $R_\mathrm{skew}$ is skew-symmetric, $R_{\ker \calL}$ is in the kernel of $\calL_X$~\eqref{eq:calL} and $R_{(\ker \calL)^\bot}$ is in the orthogonal complement of the latter in $\Spp$. Clearly, $YR_\mathrm{skew}$ and $YR_{\ker \calL}$ are tangent vectors, thus vectorized versions of these are already in the span of $U$. On the other hand, by definition, $YR_{(\ker \calL)^\bot}$ is not tangent at $Y$ (if it is nonzero). This raises $k'$ (the rank of $W$) to:
	\begin{align}
		k' & = k + p^2 - \frac{p(p-1)}{2} - \dim \calF_X = np - m\frac{d(d+1)}{2} + \frac{p(p+1)}{2} - \dim \calF_X.
		\label{eq:kprime}
	\end{align}
	Combine with $\lambda_{\floorr{(np - k')/p}} \geq 0$ and the definition of $\Delta$~\eqref{eq:dimF} to conclude.
\end{proof}
Theorem~\ref{thm:eigS}
is particularly meaningful for linear $f$, considering the intuition that for $p \geq p^*$, generically, $\dim \calF_X = \Delta \geq 0$~\eqref{eq:dimF} (Theorem~\ref{thm:genericfacedimension} gives a proof for $d=1$). Thus, for such $p$, a second-order critical point is either globally optimal, or it maps to a face of abnormally high dimension, over which $f$ must be constant and suboptimal. We could not produce an example of the latter.
We summarize this in a corollary, followed by a question.
\begin{corollary}
	Assume $f$ is linear and fix $p > p^*$, hence $\Delta > 0$. If $Y\in\Stdpm$ is a second-order critical point for~\eqref{eq:RP} but $X = YY\transpose$ is not a global optimizer for~\eqref{eq:P}, then $\rank(X) = p$, $\dim \calF_X \geq \Delta + p$ and $f$ is constant over $\calF_X$.
	\label{cor:linearfconstantface}
\end{corollary}
\begin{proof}
	Use Corollary~\ref{cor:concaveffacerestriction}, Theorem~\ref{thm:eigS} and Theorem~\ref{thm:rankdeficientY}.
\end{proof}
For linear $f(X) = \innersmall{C}{X}$ and $p > p^*$, the question is the following: if $C$ is sampled uniformly at random from the unit-norm symmetric matrices, what is the probability that $f$ is constant over a face $\calF_X$ of dimension $\Delta + p$ or larger, with $\rank(X) = p$? \emph{If it is zero, then almost surely all second-order critical points of~\eqref{eq:RP} are global optimizers.} We do not answer this question here, but refer to Theorem~\ref{thm:genericfacedimension} to argue that there are few such faces.

Theorem~\ref{thm:eigS} is motivation to investigate upper-bounds on the dimensions of faces of $\calC$.
The following result extends~\citep[Thm.\,3.1(i)]{laurent1996facial} to $d\geq 1$.
\begin{theorem}\label{thm:dimFbounds}
	If $X\in\calC$ has rank $p$, then the face $\calF_X$~\eqref{eq:calF} has dimension bounded as:
	\begin{align}
		\frac{p(p+1)}{2} - n\frac{d+1}{2} \ \leq \  \dim \calF_X \ \leq \  \frac{p(p+1)}{2} - p\frac{d+1}{2}.
	\end{align}
	If $p$ is an integer multiple of $d$, the upperbound is attained for some $X$.
\end{theorem}
\begin{proof}
	Inequality~\eqref{eq:dimF} covers the lower bound. It remains to show that $\calL_X(A) = 0$ imposes at least $p(d+1)/2$ linearly independent constraints on $A\in\Spp$. Let $Y\in\Stdpm$ be such that $X = YY\transpose$, and let $y_1, \ldots, y_n \in \Rp$ denote the rows of $Y$, transposed. 
	Greedily select $p$ linearly independent rows of $Y$, in order, such that row $i$ is picked iff it is linearly independent from rows $y_1$ to $y_{i-1}$. This is always possible since $Y$ has full rank. Write $t = \{ t_1 < \cdots < t_p \}$ to denote the indices of selected rows. Write $s_k = \{ ((k-1)d+1), \ldots, kd \}$ to denote the indices of rows in slice $Y_k$, and let $c_k = s_k \cap t$ be the indices of selected rows in that slice.
	
	For $x_1,\ldots,x_p\in\Rp$ linearly independent, the $p(p+1)/2$ symmetric matrices $x_i^{}x_j\transpose + x_j^{}x_i\transpose$ form a basis of $\Spp$---see for example~\citep[Lem.\,2.1]{laurent1996facial}. Defining $E_{ij} = y_i^{}y_j\transpose + y_j^{}y_i\transpose = E_{ji}$, this means $\calE_t = \{ E_{t_{\ell^{}},t_{\ell'}} : \ell,\ell' = 1\ldots p \}$ forms a basis of $\Spp$. Similarly, since each slice $Y_k$ has orthonormal rows, the matrices $\{E_{ij} : i,j\in s_k\}$ are linearly independent.
	
	The constraint $\calL_X(A) = 0$ means $\inner{A}{E_{ij}} = 0$ for each $k$ and for each $i,j \in s_k$. To establish the theorem, we need to extract a subset $T$ of at least $p(d+1)/2$ of these $md(d+1)/2$ constraint matrices, and guarantee their linear independence. To this end, let
	\begin{align}
		T & = \{ E_{ij} : k \in \{1,\ldots,m\} \textrm{ and } i \in c_k, j \in s_k \}.
	\end{align}
	That is, for each slice $k$, $T$ includes all constraints of  that slice which involve at least one of the selected rows. For each slice $k$, there are $|c_k|d - \frac{|c_k|(|c_k|-1)}{2}$ such constraints---note the correction for double-counting the $E_{ij}$'s where both $i$ and $j$ are in $c_k$. Thus, using $|c_1|+\cdots+|c_m|=p$, the cardinality of $T$ is:
	\begin{align}
		|T| & = \sum_{k = 1}^m |c_k|d - \frac{|c_k|(|c_k|-1)}{2} = p(d+1/2) - \frac{1}{2}\sum_{k = 1}^m |c_k|^2.
		\label{eq:contraintcollectionsize}
	\end{align}
	We first show matrices in $T$ are linearly independent. Then, we show $|T|$ is large enough.
	
	Consider one $E_{ij} \in T$: $i = t_\ell$ for some $\ell$ (otherwise, permute $i$ and $j$) and $i,j \in s_k$ for some $k$. By construction of $t$, we may expand $y_j$ in terms of the rows selected in slices 1 to $k$, i.e., $y_j = \sum_{\ell' = 1}^{\ell_k} \alpha_{j,\ell'} y_{t_{\ell'}}$, where $\ell_k = |c_1| + \cdots + |c_k|$. As a result, $E_{ij}$ expands in the basis $\calE_t$ as follows: $E_{ij} = \sum_{\ell' = 1}^{\ell_k} \alpha_{j,\ell'} E_{t_\ell,t_{\ell'}}$. As noted before, $E_{ij}$'s contributed by a same slice $k$ are linearly independent. Furthermore, they expand in only a subset of the basis, namely: $\calE_t^{(k)} = \{ E_{t_\ell, t_{\ell'}} : \ell_{k-1} < \ell \leq \ell_k, \ell' \leq \ell_k \}$. For $k\neq k'$, $\calE_t^{(k)}$ and $\calE_t^{(k')}$ are disjoint.
	Hence, elements of $T$ are linearly independent.

	It remains to lowerbound~\eqref{eq:contraintcollectionsize}. To this effect, use $|c_k| \leq d$ to obtain:
	\begin{align*}
		\sum_{k = 1}^m |c_k|^2 \leq \max_{x\in\Rm : \|x\|_\infty \leq d, \|x\|_1 = p} \|x\|_2^2 = \floorr*{\frac{p}{d}}d^2 + \left( p - \floorr*{\frac{p}{d}}d \right)^2 \leq pd.
	\end{align*}
	Indeed, the maximum is attained by making as many of the entries of $x$ as large as possible---this can be verified using KKT conditions. %
	In combination with~\eqref{eq:contraintcollectionsize}, this confirms at least $p(d+1/2)-pd/2 = p(d+1)/2$ linearly independent constraints act on $A$, thus upperbounding $\dim \calF_X$.
	
	To conclude, we argue that the proposed upperbound is essentially tight. Indeed, build $Y$ by repeating $m$ times the $d$ first rows of $I_p$, then by replacing its $p$ first rows with $I_p$ (to ensure $Y$ is full-rank). If $p/d$ is an integer, then exactly the $p/d$ first slices each contribute $d(d+1)/2$ independent constraints, i.e., $\dim \calF_{YY\transpose} = p(p+1)/2 - p(d+1)/2$.
\end{proof}
Theorems~\ref{thm:KKTS},~\ref{thm:eigS} and~\ref{thm:dimFbounds} combined give a sufficient condition on $p$ to ensure all second-order critical points of~\eqref{eq:RP} correspond to KKT points of~\eqref{eq:P}.
\begin{corollary}\label{cor:sufficientp}
	Assume $f$ is concave and $p > \frac{d+1}{d+3}n$. If $Y\in\Stdpm$ is a second-order critical point for~\eqref{eq:RP}, then $X = YY\transpose$ is a KKT point for~\eqref{eq:P}. If furthermore $f$ is linear, then all second-order critical points of~\eqref{eq:RP} are global optimizers.
\end{corollary}
\begin{proof}
	Since $\rank(X) \leq p$, we have $\dim \calF_X - \Delta \leq (n-p)\frac{d+1}{2}$. Theorem~\ref{thm:eigS} then gives $S \succeq 0$~\eqref{eq:S} if $(n-p)(d+1) < 2p$, which is the case. Apply Theorem~\ref{thm:KKTS} to conclude.
\end{proof}
In particular, for the Max-Cut SDP ($d = 1$, $f$ linear), this shows that computing a second-order critical point of~\eqref{eq:RP} with $p = \floorr*{n/2} + 1$ certainly solves~\eqref{eq:P}. This is an interesting and new result, but of course, in practice, it is desirable (and empirically sufficient) to take $p = \floorr*{p^*} + 1$ (much smaller). In the unlikely event we would encounter a ``bad'' second-order critical point with such $p$, the following theorem provides an escape route (for concave $f$) which does not require increasing the rank. It proceeds by moving inside a face.
\begin{proposition}[in-face rank reduction]
	Let $Y\in\Stdpm$ have full-rank, $X = YY\transpose$, and consider the symmetric operator $\calH$ on $\Spp$ defined by $\calH(A) = Y\transpose \sbd{YAY\transpose}Y$. $\calH$ is positive semidefinite and $\dim \ker \calH = \dim \calF_X$. If $A \in \ker \calH$ is nonzero, then $X' = Y(I_p - A/\lambdamin(A))Y\transpose \in \calF_X$ (on the boundary) and $\rank(X') \leq p-1$.
	\label{prop:facerankreduction}
\end{proposition}
\begin{proof}
	Recall the definitions of $\calF_X$~\eqref{eq:calF} and $\calL_X$~\eqref{eq:calL}. All follows from $\calH = \calL_X^* \calL_X^{}$, where $\calL_X^*$ is the adjoint of $\calL_X$.
\end{proof}
The latter proposition suggests an explicit numerical method to compute $A$, by computing a minimal eigenvector of $\calH$. Applying $\calH$ costs $\calO(m(d^2p + p^2d))$ flops. Assuming $p = \floorr{p^*}+1 = \Theta(d\sqrt{m})$ and that up to $p(p+1)/2$ applications are necessary, this brings the cost of computing $A$ to $\calO(d^2 n^3)$ flops. %

%
%
%
%
%
%
%
%
%
%

%

%
%
%
%
%
%

%
%
%
%
%
%
%
%
%
%
%
%

\section{The Riemannian staircase algorithm} \label{sec:algorithm}

The above results suggest a simple algorithm to compute KKT points of~\eqref{eq:P}:  for some small value of $p \geq d+1$, find a second-order critical point $Y$ of~\eqref{eq:RP}. If $Y$ is rank deficient, then Theorem~\ref{thm:rankdeficientY} guarantees $X = YY\transpose$ is KKT for~\eqref{eq:P}. Otherwise, increase $p$ and find a second-order critical point of~$(\text{RP}_{p_+})$, possibly warm-starting as suggested by Corollary~\ref{cor:escapefullrank}. Iterating this procedure, the worst-case scenario is when $p$ increases all the way to $n$, in which case any second-order critical point of~$(\text{RP}_n)$ yields a KKT point of~\eqref{eq:P}, as per Theorem~\ref{thm:squareY}. Specific results pertaining to classes of functions $f$ limit how large $p$ could grow. We call this the \emph{Riemannian Staircase}, listed as Algorithm~\ref{algo:staircase}. Of course, the hope is that the algorithm returns for some small $p$, and in practice we find that it is often sufficient to take $p$ just above the rank of a solution.

\begin{algorithm}[t]
	\caption{Riemannian Staircase Algorithm}
	\label{algo:staircase}
	\begin{algorithmic}[1]
		\State \textbf{Input:} Integers $d < p_1 < p_2 < \cdots < p_k \leq n$; an initial iterate $Y_0 \in \St(d, p_1)^m$.
		\For {$i = 1 \ldots k$ }
		\State $Y_i \leftarrow \Call{RiemannianOptimization}{\St(d, p_i)^m, g, Y_{i-1}}$ \Comment{Descent to 2nd order critical}
		\If {$i = k$ \textbf{or} $\rank(Y_i) < p_i$}
		\State \textbf{return $Y_i$} \Comment{Theorems~\ref{thm:rankdeficientY} and~\ref{thm:squareY}}
		\Else
		\State $Y_i \leftarrow \begin{pmatrix} Y_i & 0_{n\times (p_{i+1}-p_i)} \end{pmatrix}$ \Comment{Augment $Y_i$ for the next rank}
		\State $Z \leftarrow \Call{EscapeDirection}{\St(d, p_{i+1})^m, g, Y_i}$ \Comment{Corollary~\ref{cor:escapefullrank} + line-search}
		\State $Y_i \leftarrow \Retr_{Y_i}(Z)$ \Comment{Eq.~\eqref{eq:retraction}}
		\EndIf
		\EndFor
	\end{algorithmic}
\end{algorithm}

Algorithm~\ref{algo:staircase} assumes availability of a procedure $\Call{RiemannianOptimization}{\M, g, Y_0}$, which returns a second-order critical point of $g \colon \M \to \reals$, with cost at most $g(Y_0)$. This assumption is discussed below.

Inside the \emph{else}-block, the augmented $Y_i$ (with additional columns of zeros) is (usually) a saddle point. Although the second-order procedure $\Call{RiemannianOptimization}{}$ should be able to escape it, we make this step explicit via the procedure $\Call{EscapeDirection}{}$. The latter can be implemented using Corollary~\ref{cor:escapefullrank}, which indicates how computing an eigenvector of $S$~\eqref{eq:S} associated to its smallest eigenvalue, combined with a line-search, allows to escape the saddle with strict cost decrease (unless that eigenvalue is nonnegative, in which case $Z = 0$ and $Y_i$ is returned with $Y_i^{} Y_i\transpose$ being KKT). %

For all sufficiently smooth $f$,
taking $p_k = n$ guarantees Algorithm~\ref{algo:staircase} returns $Y$ such that $YY\transpose$ is a KKT point. For convex $f$, KKT points may have arbitrary rank, so that allowing large $p_k$ seems necessary in general. For strongly concave $f$, it is sufficient to take $p_k = \floor{p^*} + 1$ (Corollary~\ref{cor:concaveffacerestriction}) ; for concave (and linear) $f$, it is sufficient to take $p_k = \floor{\frac{d+1}{d+3}n} + 1$ (Corollary~\ref{cor:sufficientp}), and it is expected that $p_k = \floor{p^*} + 1$ should be sufficient (Corollary~\ref{cor:linearfconstantface} and discussion).

In the latter case, in the unlikely event that Algorithm~\ref{algo:staircase} terminates with $Y$ of size $n \times p_k$, $p_k \geq \floor{p^*} + 1$, full-rank and second-order critical such that $X = YY\transpose$ is \emph{not} a KKT point of~\eqref{eq:P}, it is possible to further optimize without increasing the rank. Indeed, since $\dim \calF_X > 0$, Proposition~\ref{prop:facerankreduction} shows how to compute $Y'$ such that $X' = Y'(Y')\transpose$ is on the boundary of $\calF_X$. Since $f$ is concave, $f(X') \leq f(X)$ (Lemma~\ref{lem:faces}). $Y'$ is critical and rank-deficient. If $Y'$ is second-order critical, $X'$ is KKT. Otherwise, Theorem~\ref{thm:escaperankdeficient} shows how to escape with a strict cost decrease. Iterating this procedure as needed, the cost decreases strictly (no cycling), and the rank $p$ never exceeds $p_k$. We expect this procedure to terminate since~\eqref{eq:P} admits KKT points of rank at most $\floorr{p^*}$, but we do not prove this. %

In practice, for the $\Call{RiemannianOptimization}{}$ procedure, we use the Riemannian trust-region method (RTR)~\citep{genrtr}, through the Manopt toolbox~\citep{manopt}. RTR is a descent method. It converges toward critical points regardless of the initial iterate (global convergence).\footnote{If the local optimizers of $g$ were isolated, we could also guarantee local convergence at a quadratic rate, but $g(Y) = g(YQ)$ for all orthogonal $Q$, so this is never the case. In practice though, we do observe a characteristically superlinear convergence.
} Furthermore, the stable fixed points of RTR are local optimizers, thus making convergence to points which are not second-order critical unlikely (but not impossible). Should this happen, Theorem~\ref{thm:escaperankdeficient} shows how to escape. Admittedly, it is unclear how many times this might have to be repeated in the worst case.

%
%

Ideally, one would modify the RTR algorithm itself to ensure global convergence to second-order critical points. To the best of our knowledge, algorithms with such properties have not yet been described in the Riemannian setting. Nevertheless, we are hopeful that this should be possible, in the light of recent work by \citet{cartis2012complexity}. These authors indeed describe a modification of the classical trust-region method
and guarantee polynomial-time convergence to approximate second-order critical points.
Encouragingly, \citet{sun2015complete} achieved a strong result in this vein for dictionary learning with RTR on a sphere, hinting to a possible generalization on manifolds.

RTR terminates once the norm of $\grad\,g$ drops below a certain threshold. Thus, the returned $Y$ is not exactly a critical point, and as a result it is not, in general, exactly rank deficient even when it should be.
Numerically,
we declare rank-deficiency when the condition number of $Y\transpose Y^{}$ exceeds some large threshold (say, $10^{10}$).

If a solution of rank $q$ is sought but the obtained solution $Y_p\in\Stdpm$ has rank $p > q$, one heuristic is to project $Y_p$ to $\St(d,q)^m$ with any reasonable algorithm (call it $Y_q$)---for example, compute the thin SVD of $Y_p = U\Sigma V\transpose$, retain only the first $q$ columns of $U\Sigma$ and orthonormalize each $d\times q$ slice. Then, run $\Call{RiemannianOptimization}{\St(d,q)^m, g, Y_q}$. This typically returns a local optimizer of the hard problem. Experience shows the detour via the higher dimensional relaxation may help avoid bad local traps.

\section{Special case: linear cost function} \label{sec:linear}

In the important special case where $f$ is a linear function
$f(X) = \innersmall{C}{X}$ for some data matrix $C\in\Snn$,
the convex problem \eqref{eq:P} is an SDP. As per Remark~\ref{rem:linearlargepiseasy}, it is equivalent to~\eqref{eq:RP} as soon as $p \geq p^*$.
Remarkably, for any $X\in\calC$, we obtain a lower-bound on the optimal value of the SDP, following an idea from~\citet[\S\,6.1]{burer2005local}.
\begin{proposition}[bounds on the SDP value]	\label{prop:lowerboundcostlinear}
	Let $f(X) = \innersmall{C}{X}$ be linear and let
	$f^*$ denote the optimal value of~\eqref{eq:P}. Then, for all $X\in\calC$,
	\begin{align*}
	f(X) + n\cdot\lambdamin(S(X)) \quad \leq \quad f^* \quad \leq \quad f(X).
	\end{align*}
\end{proposition}
\begin{proof}
	The dual of~\eqref{eq:P} is the following SDP:
	\begin{align*}
	\max\ \trace(C - \tilde S), \textrm{ s.t. } C - \tilde S \textrm{ is symmetric, block-diagonal, and }\tilde S \succeq 0.
	\end{align*}
	The matrix $\tilde S = S - \lambdamin(S)\cdot I_n$ for $S = S(X)$~\eqref{eq:S} is admissible. The result follows by strong duality, owing to Slater's condition.
\end{proof}
Algorithm~\ref{algo:staircase} solves
the SDP by optimizing $g$ in~\eqref{eq:RP}, whose differentials are:
\begin{align*}
	g(Y) & = \trace(Y\transpose C Y), & \nabla g(Y) & = 2CY, & \nabla^2 g(Y)[\dot Y] & = 2C\dot Y.
\end{align*}

As an illustrative example, we here apply Algorithm~\ref{algo:staircase} and competing SDP solvers to random instances of the orthogonal synchronization problem~\citep{bandeira2013approximating}. In this setting, one wishes to estimate $m$ orthogonal matrices $Q_1, \ldots, Q_m$, based on noisy measurements of the relative transformations $Q_i^{} Q_j\transpose$. See the introduction for applications.

In this benchmark, for increasing values of $m$, target matrices of size $d = 3$ are generated uniformly at random. The measurements of relative rotations are $H_{ij} = Q_i^{}Q_j\transpose + \sigma N_{ij}$ ($i < j$), where $\sigma = 0.3$ is the noise level and the $N_{ij}$'s are independent random noise matrices with i.i.d.\ normal entries. We also set $H_{ji} = H_{ij}\transpose$ and $H_{ii} = I_d$. To estimate the $Q_i$'s from the $H_{ij}$'s, we set $C = -H/(nm)$ and solve~\eqref{eq:P}. If the solution has rank $d$, this is equivalent to solving the maximum likelihood problem:
\begin{align*}
	\min_{Q_1, \ldots, Q_m \in \Od} \quad \sum_{i , j} \sqfrobnormbig{H_{ij}^{} - Q_i^{}Q_j\transpose}.
\end{align*}

Remarkably, for all instances generated, \eqref{eq:P} admits a rank $d$ solution, thus revealing the true maximum likelihood estimator: a hard quantity to compute, in general. This serendipitous phenomenon is partly explained in~\citep{bandeira2014tightness}.

Figure~\ref{fig:linear} shows how much time it takes various solvers to find this solution of rank $d$ (they all do). Algorithm~\ref{algo:staircase} runs RTR once on~\eqref{eq:RP} with $p = d+1$, with a random initial guess, and returns with an optimal rank $d$ solution.
We compare against interior point methods SeDuMi~\citep{sedumi}, SDPT3~\citep{sdpt3} and Mosek~\citep{mosek} (the latter two via CVX~\citep{cvx}) as well as against SDPLR~\citep{sdplr} with and without forcing the search rank to $d+1$ (the forced version is labeled SDPLR*). We also depict how much time it takes to simply compute the top $d$ eigenvectors of $H$, which, after projection, reveal an (empirically) equally good estimator for this problem, but with weaker guarantees~\citep{singer2010angular,bandeira2013approximating}.

\begin{figure}
\centering
\begin{tikzpicture}
\draw[draw=none,shape=rectangle]
    (0,0) node {\includegraphics[width=\linewidth]{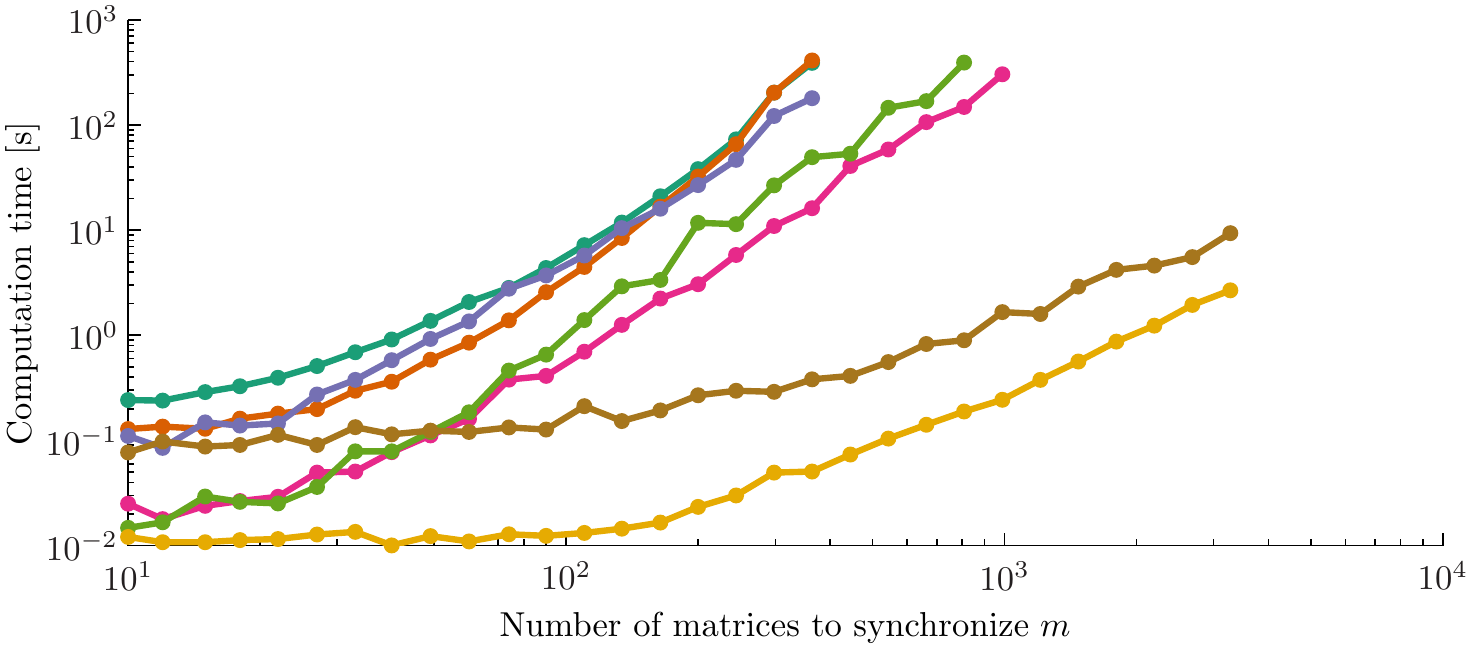}};
\draw[draw=none,shape=rectangle] (-4.5,-.0)  node[mycolor1] {SDPT3};
\draw[draw=none,shape=rectangle] (-.28,2.6)  node[mycolor2] {Mosek};
\draw[draw=none,shape=rectangle] (1.5,2.6) node[mycolor3] {SeDuMi};
\draw[draw=none,shape=rectangle] (3.6,2.6) node[mycolor4] {SDPLR};
\draw[draw=none,shape=rectangle] (3.2,3.0)   node[mycolor5] {SDPLR*};
\draw[draw=none,shape=rectangle,anchor=west] (5.3,1.1) node[mycolor6] {Staircase};
\draw[draw=none,shape=rectangle,anchor=west] (5.3,0.5) node[mycolor7] {EIG};
\end{tikzpicture}
\vspace{-3mm}
\caption{All methods solve~\protect\eqref{eq:P} on problems from Section~\protect\ref{sec:linear}. The proposed staircase algorithm is the only one to return a solution which satisfies the constraints up to machine precision. It also returns the solution $Y$ which is numerically closest to be of rank $d$. Its computational cost seems to grow at the same rate as that of merely computing $d$ dominant eigenvectors of the data matrix (EIG), thus outperforming interior point methods as well as SDPLR.} %
\label{fig:linear}
\end{figure}
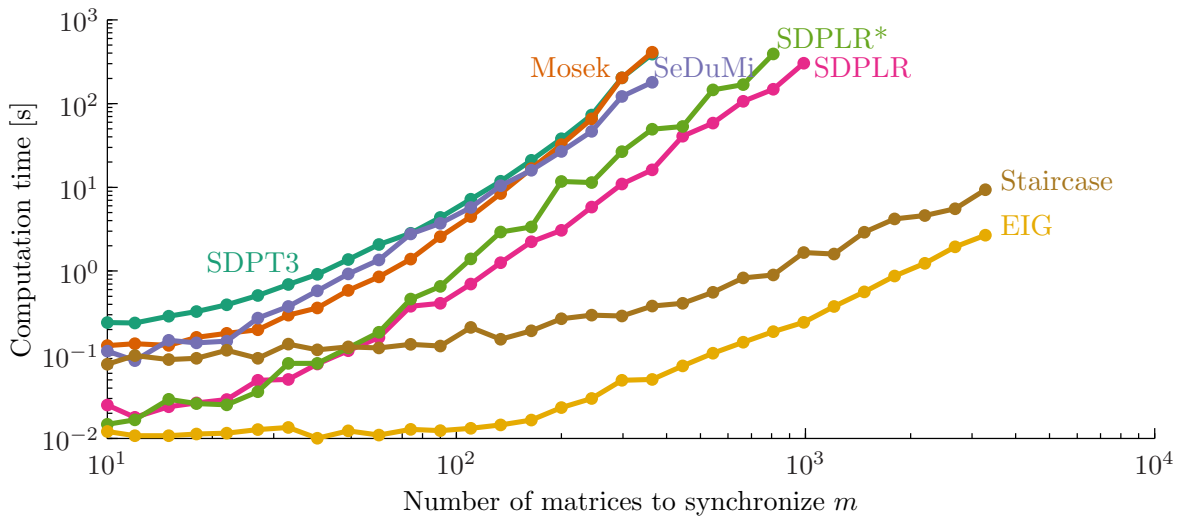

\section{Special case: Pseudo-Huber loss cost function} \label{sec:huber}

In the previous section, orthogonal synchronization is considered with Gaussian noise on the relative measurements. Maximum likelihood estimation then naturally leads to the minimization of a quadratic cost in $Y$, which simplifies to a linear cost in $X$.

When the relative measurements $H_{ij}$ include outliers, least-squares are not expected to perform well. As an alternative, \citet{wang2012LUD} minimize a sum of \emph{unsquared} errors, that is, they estimate the orthogonal matrices $Q_i$ as the minimizers of $\sum_{i,j} \smallfrobnorm{H_{ij}^{} - Q_i^{}Q_j\transpose}$.
A convex relaxation akin to the one from the previous section leads to solving~\eqref{eq:P} with the \emph{least unsquared deviations} cost $f(X) = \sum_{i,j} \frobnorm{H_{ij} - X_{ij}}$ (LUD). This is similar in spirit to the convex relaxation for robust subspace estimation presented in~\citep{lerman2014robust}. The authors show that rounding the solutions of this convex program yields a good estimator, even if up to a (random) half of the data is random. In that regime, if the non-outliers are noiseless, solving the convex program achieves perfect recovery with high probability. They solve the problem with an alternating direction augmented Lagrangian method (ADM).

Tools in this paper do not directly apply to the LUD cost, because it is nonsmooth.\footnote{Recent work on nonsmooth optimization on manifolds~\citep{kovnatsky2015MADMM} may prove useful in this regard.} Unfortunately, in our experiments we also found that smoothing the LUD cost typically leads to higher rank solutions, at a significant computational premium. We formalize this observation in the following theorem. The assumptions on $H$ are not restrictive: they require just the slightest inconsistency in the measurements.
\begin{theorem}[smoothing the LUD cost suppresses rank $d$ solutions]
Let $\ell \colon \mathbb{R}^+ \to \mathbb{R}^+$ be an increasing function (with $\ell'(x) > 0$ if $x > 0$) such that $f \colon \Snn \to \mathbb{R}$ defined by $f(X) = \sum_{i,j} \ell(\frobnorm{X_{ij} - H_{ij}})$ is twice continuously differentiable, where each $H_{ij}^{} = H_{ji}\transpose$
verifies $\opnormsmall{H_{ij}} \leq 1$ (which includes orthogonal matrices),
and $H_{ii} = I_d$. %
If $H$ is not a rank-$d$ matrix in $\calC$, then all KKT points of~\eqref{eq:P}
have rank strictly larger than $d$.
(Otherwise, $X = H$ is the unique KKT point.)
\end{theorem}
\begin{proof}
The gradient of $f$ with respect to $X_{ij}$ is given by $\nabla f(X)_{ij} = w_{ij}(X_{ij} - H_{ij})$, with $w_{ij} = w_{ji} = \ell'(\frobnorm{X_{ij} - H_{ij}})/\frobnorm{X_{ij} - H_{ij}} > 0$ if $X_{ij} \neq H_{ij}$, and $w_{ij} = 0$ otherwise. This is well defined by assumption.
For contradiction, assume $X$ is a KKT point of~\eqref{eq:P} and $\rank(X) = d$. By Theorem~\ref{thm:KKTS}, $S$~\eqref{eq:S} is positive semidefinite. In particular, its diagonal blocks are positive semidefinite:
\begin{align*}
	S_{ii} & = \nabla f(X)_{ii} - \symmop\Big(\sum\nolimits_{j} \nabla f(X)_{ij} X_{ji}\Big) \\
		   & = \sum\nolimits_{j\neq i} w_{ij} \symm{H_{ij}^{}X_{ij}\transpose - X_{ij}^{}X_{ij}\transpose} \\
		   & = \sum\nolimits_{j\neq i} w_{ij} \symm{H_{ij}^{}X_{ij}\transpose - I_d} \succeq 0.
\end{align*}
The last equality follows from the fact that, since $\rank(X) = d$, each $X_{ij}$ is orthogonal (Proposition~\ref{prop:Xijconvhull}). Since $\opnormsmall{H_{ij}} \leq 1$, each term $\symmsmall{H_{ij}^{}X_{ij}\transpose - I_d}$ is \emph{negative} semidefinite. Thus, the $S_{ii}$'s are zero (simultaneously positive and negative semidefinite), showing that $S = 0$ (by Schur's complement). Hence, the off-diagonal blocks are zero too: $S_{ij} = \nabla f(X)_{ij} = w_{ij}(X_{ij} - H_{ij}) = 0$, implying $X = H$: a contradiction.
\end{proof}
\begin{remark}
The latter theorem applies in particular for $\ell(x) = x^2$. Thus, the convex problem~\eqref{eq:P} with $f(X) = \sqfrobnorm{X-H}$ is not expected to admit rank $d$ solutions in the presence of even the smallest noise. This is in sharp contrast with the linear case, $f(X) = -\Trace(HX)$, even though these two costs differ only by a constant over the rank-$d$ feasible $X$'s. %
The key difference is that the linear cost is also concave, pointing to concavity and nonsmoothness to promote rank-$d$ solutions. %
\end{remark}

In view of these results, we take interest in minimizing the related smoothed cost:
\begin{align}
	g(Y) & = f(YY\transpose) = \sum_{i,j} \ell_\varepsilon\left(\frobnorm{H_{ij}Y_j - Y_i}\right), & \ell_\varepsilon(x) & = %
	\sqrt{x^2 + \varepsilon^2} - \varepsilon \to |x| \textrm{ as } \varepsilon \to 0.
	\label{eq:huber}
\end{align}
Although it bears much resemblance with the convex LUD cost (they coincide when $\rank(X) = d$ and $\varepsilon = 0$), this $f$ is \emph{strongly concave} in $X$. Indeed,
$\sqfrobnorm{H_{ij}Y_j - Y_i}$ is affine in $X_{ij}$,
so that $f$ is a sum of square roots of affine functions of $X$, and the terms under the square roots are larger than $\varepsilon^2 > 0$ and smaller than $\varepsilon^2 + \left(\frobnormsmall{H_{ij}} + \sqrt{d}\right)^2$.
The following marks the dependence in $X$ more explicitly:
\begin{align*}
	f(X) & = \sum_{i,j} \sqrt{\sqfrobnormsmall{H_{ij}} + \sqfrobnormsmall{I_d} - 2\inner{H_{ij}}{X_{ij}} + \varepsilon^2} - \varepsilon, \\
	\nabla f(X)_{ij} & = \frac{-1}{\sqrt{\sqfrobnormsmall{H_{ij}} + \sqfrobnormsmall{I_d} - 2\inner{H_{ij}}{X_{ij}} + \varepsilon^2}} H_{ij}.
\end{align*}
The only difference with a smoothed LUD cost is the term $\sqfrobnormsmall{I_d}$ which appears instead of $\sqfrobnormsmall{X_{ij}}$. Considering that the aim is for the $X_{ij}$'s to be orthogonal, which maximizes their norm by Proposition~\ref{prop:rankdmaxnorm}, refraining from minimizing $\sqfrobnormsmall{X_{ij}}$ appears as a good start.

We may still compute a KKT point for~\eqref{eq:P}, but there is no guarantee that such a point will be even a local minimizer anymore.
On the bright side,
Corollary~\ref{cor:concaveffacerestriction} states that, by strong concavity of $f$, all KKT points of~\eqref{eq:P} are extreme points (thus they have rank at most $\floorr{p^*}$~\eqref{eq:patakirank}) and, for $p > p^*$, all second-order critical points of~\eqref{eq:RP} reveal KKT points of~\eqref{eq:P}.
The numerical experiment below shows that, empirically, even for $\varepsilon > 0$, the proposed algorithm typically converges to a rank-$d$ KKT point of excellent quality. %
Furthermore, as $\varepsilon$ is decreased, the quality of the found KKT point increases (with warm-starting).

We now use the proposed robust formulation of orthogonal synchronization to situations where the sought matrices are in fact permutations\footnote{Permutation matrices are binary matrices with exactly one 1 on each row and column. They are orthogonal.} (without modifying the algorithms). Synchronization of permutations notably arises in image association problems in computer vision~\citep{huang2013consistent,pauchauri2014permutation}.

Let the $Q_i$'s be permutations to estimate and let the $H_{ij}$'s be measurements of the relative permutations $Q_i^{}Q_j\transpose$.
A subset of the measurements of a given size is selected uniformly at random and replaced by uniformly random permutations (outliers). The other measurements are correct.
If perfect recovery of the $Q_i$'s is achieved, then the permutations are recovered. %

Figure~\ref{fig:huber} exhibits the perfect recovery phenomenon hinted by~\citet{wang2012LUD}. We say ``hinted'' as the chosen scenario does not exactly fit the assumptions of these authors. Even in the face of many outliers, the true permutations are recovered, showing the applicability of the proposed methods to permutation estimation.

\begin{figure}
	\centering
	\begin{tikzpicture}
	\draw[draw=none,shape=rectangle]
	(0,0) node {\includegraphics[width=\linewidth]{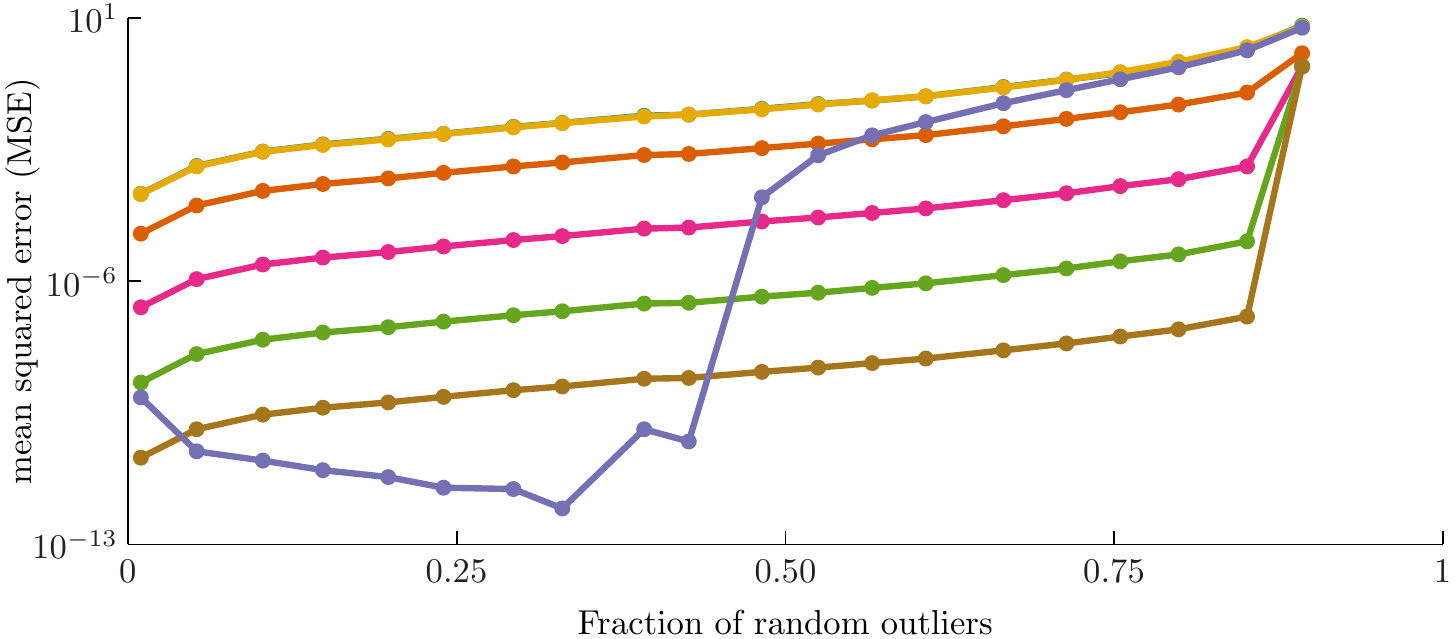}};
	\draw[draw=none,shape=rectangle,anchor=west] (-5,2.4)   node[mycolor1] {Staircase linear cost};
	\draw[draw=none,shape=rectangle,anchor=west] (-1.1,-1.6)   node[mycolor3] {ADM};
	\draw[draw=none,shape=rectangle,anchor=west] (2.6,-.60) node[mycolor6] {$\varepsilon = 10^{-3}$};
	\draw[draw=none,shape=rectangle,anchor=west] (2.6,-1.05) node[mycolor6] {Staircase pseudo-Huber};
	\draw[draw=none,shape=rectangle,anchor=west] (2.6,.25) node[mycolor5] {$\varepsilon = 10^{-2}$};
	\draw[draw=none,shape=rectangle,anchor=west] (2.6,1.00) node[mycolor4] {$\varepsilon = 10^{-1}$};
	\draw[draw=none,shape=rectangle,anchor=west] (2.6,1.75) node[mycolor2] {$\varepsilon = 10^{0}$};
	\draw[draw=none,shape=rectangle,anchor=west] (-1,2.5)     node[mycolor7] {EIG};
	\end{tikzpicture}
	\vspace{-3mm}
	\caption{Synchronization of $m = 100$ permutations of size $d = 6$. As explained in Section~\protect\ref{sec:huber}, for each pair of permutations we are given a relative permutation measurement. Some fraction of those are exact: this varies on the horizontal axis. The other measurements (selected uniformly at random) are uniformly random. When the mean squared error (vertical axis) is close to zero (say, below $10^{-6}$), the estimation is essentially perfect and we get back the true permutations. Remarkably, the staircase method with the pseudo-Huber loss cost~\protect\eqref{eq:huber} can accommodate up to 80\% of outliers and still (empirically) achieve perfect recovery. It is faster and appears more resilient than ADM~\citep{wang2012LUD}, but unfortunately, without access to the ground truth, we cannot claim we found a global optimum because~\eqref{eq:huber} is concave. ADM, on the other hand, comes with guarantees as it solves a convex problem.
	}
	\label{fig:huber}
\end{figure}
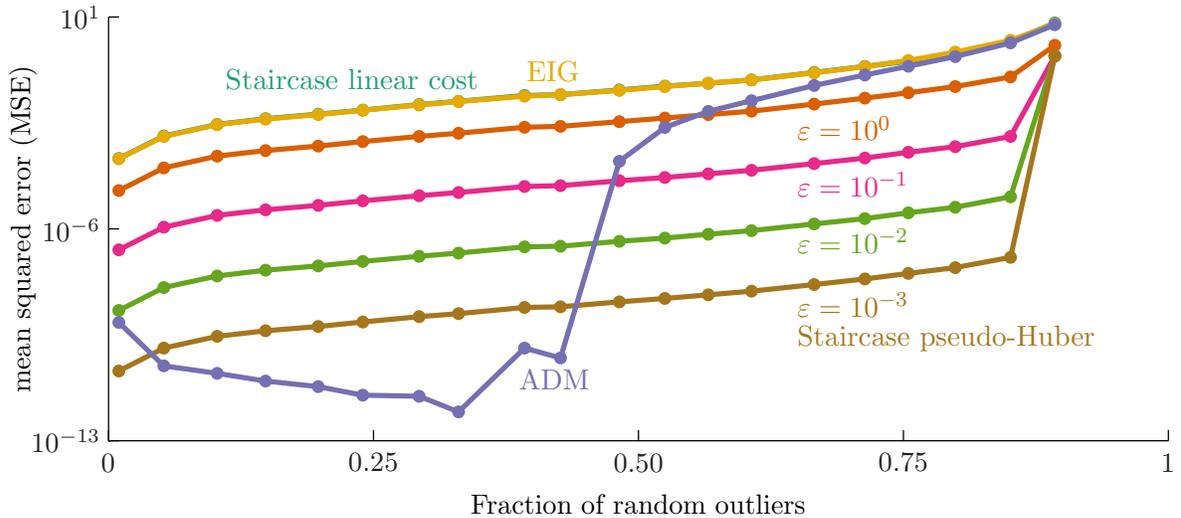

In practice, we minimize $f$ for some starting value $\varepsilon = 1$, then re-solve for decreasing values down to $\varepsilon = 10^{-3}$, warm-starting each new solve with the previous solution. The staircase method starts with a search rank $p = d+1$. For up to 80\% outliers, RTR converges to a rank-$d$, second-order critical point of $g$ ($\sigma_{d+1}(Y) \approx 10^{-10}$, $\|\grad\,g(Y)\| \leq 10^{-6}$ and $\lambda_{\textrm{min}}(\operatorname{Hess} g(Y)) \geq -10^{-10}$, after scaling
) without the need to increase $p$, thus rapidly identifying a KKT point of~\eqref{eq:P} which appears to be a global optimizer. %
We compare with ADM~\citep{wang2012LUD} optimizing the LUD cost, but not with the IPM's, as they rapidly run out of memory.

\section{Conclusions and perspectives}

We proposed a novel algorithm to compute KKT points for optimization problems over a class of spectrahedra that come up in relaxations of  various problems involving orthonormal matrices.
Our approach consists in exploiting the smooth geometry of bounded-rank subsets of those spectrahedra, to reduce the problem to Riemannian optimization. This effectively allows one to control how much lifting (dimension increase) is involved in the relaxation. An investigation of both the convex and the Riemannian geometries of the total and the bounded-rank problem showed that, under certain conditions, it is only necessary to compute second-order critical points on a low-dimensional portion of the boundary of the spectrahedron. Numerical experiments confirm the usefulness of this observation.

The present work triggers a number of questions for future investigation.
\begin{itemize}
	\item Which spectrahedra are such that their elements of bounded rank form a smooth manifold? (\citet{journee2010low} cover a number of such sets.) When the search space is of such form with additional constraints, can those be accommodated efficiently? This would be useful to address the SDP's in, e.g.,~\citep{huang2013consistent,saunderson2014semidefinite,chen2014near,bandeira2015nonuniquegames}.
	\item What is the computational complexity of obtaining a second-order critical point of a sufficiently smooth function on a Riemannian manifold, up to a given accuracy? This might be answered by following work in~\citep{cartis2012complexity,sun2015complete}. When $Y$ is only approximately second-order critical and rank deficient, is $YY\transpose$ approximately KKT, in a certain sense? For linear $f$, Proposition~\ref{prop:lowerboundcostlinear} offers a positive answer.
	\item For nonconvex $f$, the set of KKT points includes the local optimizers of~\eqref{eq:P}, as well as a number of uninteresting points. All KKT points give rise to critical points, but not necessarily second-order critical points. Can this be used to improve guarantees? Could we compute second-order KKT points instead, thus possibly excluding even more spurious points? A starting point might be \citep[Thm.\,3.45]{ruszczynski2006nonlinear} and \citep{shapiro1997secondorder}.
	\item Assuming linear $f$, if~\eqref{eq:P} admits a unique solution of rank $r$ (see, e.g., \citep{bandeira2014tightness}), is it sufficient to explore~\eqref{eq:RP} with $p = r+1$? Under the noise model of Section~\ref{sec:linear}, this is observed empirically. Perhaps, this could be investigated via the expected size of the attraction basin of the global optimizers, similarly to~\citep{sun2015complete} in the context of dictionary learning.
	\item Finally, regarding Corollary~\ref{cor:linearfconstantface} and its attached question: for a random linear cost function and $p > p^*$, what is the probability that \eqref{eq:RP} admits second-order critical points which are not global optimizers?
\end{itemize}

\subsubsection*{Acknowledgments}

The author thanks P.-A.~Absil, A.~d'Aspremont, A.~Bandeira, X.~Cheng, B.~Gerencs\'er, Y.~Khoo, B.~Mishra, A.~Singer and B.~Vandereycken for fruitful discussions.
Parts of this research were conducted while N.B.\ was a research fellow with the FNRS in Belgium, and while generously supported by a Research in Paris grant, by the ``Fonds Sp\'eciaux de Recherche'' (FSR) at UCLouvain and by the Chaire Havas ``Chaire Eco\-no\-mie et gestion des nouvelles don\-n\'ees'', the ERC Starting Grant SIPA and a Research in Paris grant in France.

{\small
\bibliographystyle{plainnat}
\bibliography{boumal}
}

\appendix

\section{Tightness for synchronization on a cycle}
\label{sec:cycletight}

Consider synchronization on a cycle: the goal is to estimate orthogonal matrices $R_1, \ldots, R_m \in \Od$ based on one cycle of measurements: $H_{1,2} \approx R_1^{}R_2\transpose, H_{2,3} \approx R_2^{}R_3\transpose, \ldots, H_{m,1} \approx R_m^{}R_1\transpose$. This is achieved by solving~\eqref{eq:RP} with $p=d$ and
$$
g(Y) = \sum_{i=1}^m \sqfrobnormsmall{Y_i^{}Y_{i+1}\transpose - H_{i,i+1}^{}} = \sum_{i=1}^{m} \innersmall{Y_i^{}Y_{i+1}\transpose}{-H_{i,i+1}^{}} + \textrm{constant},
$$
with the indexing convention that $Y_{m+1} \equiv Y_1$ and $H_{m,m+1} \equiv H_{m, 1}$.
Under a permissive condition on the measurements, \citet{sharp2004multiview,peters2014sensor} exhibit an explicit formula for the solution (they restrict their attention to rotation matrices, that is, orthogonal matrices of determinant +1). We show that under that same condition, the corresponding SDP relaxation~\eqref{eq:P} with
\begin{align}
f(X) & = \inner{C}{X}, & C & =
-\begin{pmatrix}
0 & H_{1,2} & & & H_{m,1}\transpose \\
H_{1,2}\transpose & 0 & H_{2,3} & & \\
& H_{2,3}\transpose & \ddots & \ddots & \\
& & \ddots & \ddots & H_{m-1,m} \\
H_{m,1} & & & H_{m-1,m}\transpose & 0
\end{pmatrix}
\label{eq:fcycle}
\end{align}
is tight: there exists a unique solution of rank $d$ which reveals the global optimum.

The proof rests on two key ingredients: (a) we have an explicit formula for the solution $X$ to certify, and (b) we have an explicit formula for the dual certificate $S(X)$ to check. It seems reasonable to expect that the result should carry over to connected graphs whose cycles have disjoint edges.

Recently, \citet{zhang2015disentangling} gave a powerful tightness result for such SDP's: as stated, their result is restricted to cycles of length 3, but it nicely accommodates non-orthogonal blocks in the data matrix $C$.

\begin{theorem}
	Let $H_{1,2}, H_{2,3}, \ldots, H_{m,1} \in \Od$ represent orthogonal measurements on a cycle ($m\geq 3$) %
	and define their product $P = H_{1,2}\cdot H_{2,3} \cdots H_{m,1} \in \Od$. If $-1$ is not an eigenvalue of $P$, then the semidefinite program~\eqref{eq:P} with cost~\eqref{eq:fcycle} admits a unique solution of rank $d$.
\end{theorem}
\begin{proof}
	\emph{Part 1: guessing $X$.}
	When the measurements are perfectly consistent, $P = I_d$, and it is easy to construct $X$: set $Y_m = I_d$ and $Y_{i} = H_{i,i+1} Y_{i+1}$ for $i = 1\ldots m-1$; then $X_{ij} = Y_i^{}Y_j\transpose$. This construction does not use $H_{m,1}$ but still achieves $g(Y) = 0$ owing to $P = I_d$, hence $X$ is optimal. When the cycle is inconsistent, it is reasonable to guess that the least-squares criterion will attempt to spread the inconsistency evenly over each edge~\citep{sharp2004multiview,peters2014sensor}. One $m$th of the error is represented by an offset $P^{1/m}$---taking this principal matrix root requires $P$ not to have negative eigenvalues. We build $Y$ by incorporating part of the error at each step, appropriately aligned. First define this recurrence: $Q_m = H_{m,1}$ and $Q_{i} = H_{i,i+1}Q_{i+1}$ for $i = (m-1)\ldots 1$ (note that $Q_1 = P$). Then, $Y_m = I_d$ and $Y_{i} = Q_i P^{-1/m} Q_i\transpose H_{i,i+1} Y_{i+1}$ for $i = (m-1)\ldots 1$. As previously, $X_{ij} = Y_i^{}Y_j\transpose$. It is not hard to check that $X_{ij} = Q_i^{}P^{(i-j)/m}Q_j\transpose$. Of course, $X$ is admissible for~\eqref{eq:P} with rank $d$.
	
	\emph{Part 2: certifying $X$.} By Theorem~\ref{thm:fconvexglobaloptunique}, it is sufficient to verify that $S(X)$~\eqref{eq:S} is positive semidefinite with rank $(m-1)d$. Let $U$ be a $d\times d$ unitary matrix such that $D = U^* P U$ is diagonal---$U$ always exists since $P$ is normal---and let $V = \diag(Q_1U, \ldots, Q_mU)$ be a block-diagonal unitary matrix. We use $V$ to operate a change of variables on $C$ and $X$:
	\begin{align}
	V^*CV & =
	-\begin{pmatrix}
	0 & I_d & & & D^{-1} \\
	I_d & 0 & I_d & & \\
	& I_d & \ddots & \ddots & \\
	& & \ddots & \ddots & I_d \\
	D & & & I_d & 0
	\end{pmatrix}, & (V^*XV)_{ij} & = D^{(i-j)/m}.
	\end{align}
	We used that $D$ is unitary. Indeed, since $P$ is orthogonal without 1 as an eigenvalue, its eigenvalues are such that $D = \diag(e^{i\theta_1}, \ldots, e^{i\theta_d})$ for $\theta_1, \ldots, \theta_d \in ]-\pi, \pi[$. The spectrum of $S(X)$ is identical to that of $V^*S(X)V$, thus we study:
	\begin{align}
	V^* S(X) V & = \begin{pmatrix}
	D^{1/m}+D^{-1/m} & -I_d & & & -D^{-1} \\
	-I_d & D^{1/m}+D^{-1/m} & -I_d & & \\
	& -I_d & \ddots & \ddots & \\
	& & \ddots & \ddots & -I_d \\
	-D & & & -I_d & D^{1/m}+D^{-1/m}
	\end{pmatrix}.
	\label{eq:VstarSV}
	\end{align}
	All blocks of~\eqref{eq:VstarSV} are diagonal, so that its rows and columns may be permuted (without affecting its spectrum) to make it block diagonal, with $k$th block of size $m$ given by
	\begin{align}
	A & = \begin{pmatrix}
	2\cos(\theta_k/m) & -1 & & & -e^{-i\theta_k} \\
	-1 & 2\cos(\theta_k/m) & -1 & & \\
	& -1 & \ddots & \ddots & \\
	& & \ddots & \ddots & -1 \\
	-e^{i\theta_k} & & & -1 & 2\cos(\theta_k/m)
	\end{pmatrix} = \begin{pmatrix}
	T & u \\ u^* & c
	\end{pmatrix},
	\label{eq:blockk}
	\end{align}
	with $T\in\mathbb{R}^{(m-1)\times (m-1)}$, $u \in\mathbb{C}^{m-1}$ and $c = 2\cos(\theta_k/m)$. It remains to show that $A$ is positive semidefinite with rank $m-1$ for any $\theta_k \in\, ]-\pi, \pi[$. Fortunately, $T$ is tridiagonal and Toeplitz, so that its whole spectrum is known explicitly~\citep{noschese2013tridiagonal}: $\lambda_j(T) = 2\big(\!\cos(\theta_k/m) - \cos(j\pi/m) \big)$, for $j = 1\ldots m-1$. These eigenvalues are all positive. By the Cauchy interlacing theorem,
	\begin{align}
	\lambda_1(A) \leq \lambda_1(T) \leq \lambda_2(A) \leq \lambda_2(T) \leq \cdots \leq \lambda_{m-1}(T) \leq \lambda_m(A).
	\end{align}
	In particular, $\lambda_2(A), \ldots, \lambda_m(A) > 0$. Since the vector $[e^{i(1/m)\theta_k}, e^{i(2/m)\theta_k}, \ldots, e^{i(m/m)\theta_k}]^*$ is in the kernel of $A$, it must be that $\lambda_1(A) = 0$. This concludes the proof.
\end{proof}
In general, the condition on the eigenvalues of $P$ is necessary. Indeed, for $d = 1$ and $m = 3$, choose the measurements such that $P = -1$ (for example, $+1$, $+1$ and $-1$) and verify that none of the 4 admissible rank-1 matrices are optimal. For the frequent case where the measurements are rotation matrices (that is, orthogonal with determinant +1), the condition is not too restrictive: even if they were distributed uniformly at random, $P$ would satisfy the condition almost surely.

\section{Generic face dimension}

For $d=1$, the following theorem shows almost all faces of $\calC$ have minimal dimension, as per the bound~\eqref{eq:dimF}~\citep{pataki1998rank,barvinok1995problems}. Key parts of the proof are due to Xiuyuan Cheng and Bal\'azs Gerencs\'er.
\begin{theorem}[Generically, faces have minimal dimension]
	For $d=1$,
	if $Y \in \Stdpm$ is selected uniformly at random, then, almost surely, $\dim \calF_{YY\transpose} = \max\left(0, \Delta\right)$, with $\Delta = \frac{p(p+1)}{2} - m\frac{d(d+1)}{2}$~\eqref{eq:calF}.
	\label{thm:genericfacedimension}
\end{theorem}
We first provide a useful lemma.
\begin{lemma}
	Let $v_1, \ldots, v_n$ be statistically independent random vectors in a vector space $V$ of dimension $k$. If for all $i$ and for all subspaces $U\subset V$ with $\dim U < k$, $\Pr[v_i \in U] = 0$, then, almost surely, $\dim \spann \{v_1, \ldots, v_n\} = \min(n, k)$ (which is maximal).
	\label{lem:independentvectorspan}
\end{lemma}
\begin{proof}
	The proof is by recurrence. Define $U_t = \spann\{v_1, \ldots, v_t\}$ for $t \in \{0, \ldots, n\}$. Clearly, $\dim U_0 = 0$. Assume $\dim U_t = \min(t, k)$ almost surely (a.s.). If $t \geq k$, then $\dim U_{t+1} = \dim U_t = k$ (a.s.). Otherwise, since $U_t$ is statistically independent from $v_{t+1}$, by assumption, $\dim U_{t+1} = t+1$ (a.s.). Thus, for all $t$, $\dim U_{t} = \min(t, k)$ (a.s.).
\end{proof}
\begin{proof}
	Proof of Theorem~\ref{thm:genericfacedimension}.
	Let $y_1, \ldots, y_n \in \Rp$ denote the columns of $Y\transpose$. A matrix $\dot X = YAY\transpose$ with $A \in \Spp$ is parallel to the face $\calF_X$ if $\innersmall{YAY\transpose}{e_i^{} e_i\transpose} = \innersmall{A}{y_i^{} y_i\transpose} = 0$ for all $i$. We study the dimension $s$ of the space spanned by the constraint matrices $A_i = y_i^{} y_i\transpose$, since $\dim \calF_X = \frac{p(p+1)}{2} - s$. We do so with $Y$ taken uniformly at random. Thus, the $y_i$'s are sampled independently from $\mathcal{N}(0, I_p)$, then scaled to unit norm.
	The dimension $s$ does not depend on the scaling of the vectors $y_i$, so we may safely ignore it. 
	We note in passing that the proof holds for more general distributions too.
	
	We aim to apply Lemma~\ref{lem:independentvectorspan} with $V = \Spp$, $k = \frac{p(p+1)}{2}$ and $v_i = A_i$. The $v_i$'s are i.i.d., hence we omit the subscripts. To verify the lemma's condition, let $U$ be any proper subspace of $\Spp$: there exists a symmetric matrix $X \neq 0$ in the orthogonal complement of $U$. It suffices to check that $\innersmall{X}{yy\transpose} = y\transpose Xy \neq 0$ (a.s.), with $y \sim \mathcal{N}(0, I_p)$. Diagonalize $X = QDQ\transpose$ with $Q \in \Op$. Notice that $Q\transpose y$ is distributed identically to $y$. Thus, defining $D = \diag(\lambda_1, \ldots, \lambda_p)$ and $y = (y^1, \ldots, y^p)\transpose$, it suffices to check that $\sum_{j=1}^p \lambda_j (y^j)^2 \neq 0$ (a.s.). This is indeed true, since the $(y^j)^2$ are independent: their (nontrivial) linear combination has a density which is a convolution of (scaled) $\chi^2_1$ densities: this has no point mass at zero. \qedhere

\end{proof}

We expect that this result remains valid for $d \geq 1$, but we are missing a proof.

\end{document}

%% file: preamble.tex
\usepackage[ansinew]{inputenc}

\usepackage{doi}

\usepackage{amsmath}

\usepackage{mathtools}

\DeclarePairedDelimiter{\floorr}{\lfloor}{\rfloor}

\usepackage{amsfonts}
\usepackage{amsthm}
\usepackage{amssymb}

\usepackage{color}
\usepackage{bm} 
\usepackage{dsfont} 

\usepackage{graphicx}
\usepackage[lmargin=3cm,rmargin=3cm,bottom=3cm,top=3cm]{geometry}

\usepackage{tikz}
\definecolor{mycolor1}{rgb}{0.105882,0.619608,0.466667}
\definecolor{mycolor2}{rgb}{0.85098,0.372549,0.00784314}
\definecolor{mycolor3}{rgb}{0.458824,0.439216,0.701961}
\definecolor{mycolor4}{rgb}{0.905882,0.160784,0.541176}
\definecolor{mycolor5}{rgb}{0.4,0.65098,0.117647}
\definecolor{mycolor6}{rgb}{0.65098,0.462745,0.113725}
\definecolor{mycolor7}{rgb}{0.901961,0.670588,0.00784314}
\definecolor{mycolor8}{rgb}{0.4,0.4,0.4}
\definecolor{mycolor9}{rgb}{0.301961,0,0.294118}
\definecolor{mycolor10}{rgb}{0.0313725,0.25098,0.505882}

\usetikzlibrary{calc}        

\makeatletter
\newif\ifmygrid@coordinates
\tikzset{/mygrid/step line/.style={line width=0.80pt,draw=gray!80},
         /mygrid/steplet line/.style={line width=0.25pt,draw=gray!80}}
\pgfkeys{/mygrid/.cd,
         step/.store in=\mygrid@step,
         steplet/.store in=\mygrid@steplet,
         coordinates/.is if=mygrid@coordinates}
\def\mygrid@def@coordinates(#1,#2)(#3,#4){%
    \def\mygrid@xlo{#1}%
    \def\mygrid@xhi{#3}%
    \def\mygrid@ylo{#2}%
    \def\mygrid@yhi{#4}%
}
\newcommand\DrawGrid[3][]{%
    \pgfkeys{/mygrid/.cd,coordinates=true,step=1,steplet=0.2,#1}%
    \draw[/mygrid/steplet line] #2 grid[step=\mygrid@steplet] #3;
    \draw[/mygrid/step line] #2 grid[step=\mygrid@step] #3;
    \mygrid@def@coordinates#2#3%
    \ifmygrid@coordinates%
        \draw[/mygrid/step line]
        \foreach \xpos in {\mygrid@xlo,...,\mygrid@xhi} {%
          (\xpos,\mygrid@ylo) -- ++(0,-3pt)
                              node[anchor=north] {$\xpos$}
        }
        \foreach \ypos in {\mygrid@ylo,...,\mygrid@yhi} {%
          (\mygrid@xlo,\ypos) -- ++(-3pt,0)
                              node[anchor=east] {$\ypos$}
        };
    \fi%
}
\makeatother

\usepackage{psfrag}

\usepackage{algorithm}
\usepackage{algpseudocode}
\algdef{SE}[DOWHILE]{Do}{doWhile}{\algorithmicdo}[1]{\algorithmicwhile\ #1}%

\usepackage{mathrsfs} 

\usepackage[square,numbers,compress]{natbib}

\newcommand{\remove}[1]{}
\newcommand{\removesafe}[1]{}

\usepackage{subfigure}

\newcommand{\transpose}{^\top\! }

\newcommand{\inner}[2]{\left\langle{#1},{#2}\right\rangle}
\newcommand{\innersmall}[2]{\langle{#1},{#2}\rangle}
\newcommand{\innerbig}[2]{\big\langle{#1},{#2}\big\rangle}

\newcommand{\trace}{\mathrm{trace}}
\newcommand{\Trace}{\mathrm{trace}}
\newcommand{\spann}{\mathrm{span}}

\newcommand{\symmop}{\operatorname{sym}}
\newcommand{\symm}[1]{\symmop\!\left( #1 \right)}
\newcommand{\symmsmall}[1]{\symmop( #1 )}

\newcommand{\sbdop}{\operatorname{symblockdiag}}
\newcommand{\sbd}[1]{\sbdop\!\left({#1}\right)}
\newcommand{\sbdsmall}[1]{\sbdop({#1})}

\newcommand{\Stdpm}{\St(d, p)^m}
\newcommand{\Proj}{\mathrm{Proj}}

\newcommand{\Retr}{\mathrm{Retraction}}

\newcommand{\St}{\mathrm{St}}

\newcommand{\T}{\mathrm{T}}

\newcommand{\M}{\mathcal{M}}

\newcommand{\Od}{{\mathrm{O}(d)}}
\newcommand{\Op}{{\mathrm{O}(p)}}

\newcommand{\Snn}{{\mathbb{S}^{n\times n}}}

\newcommand{\Spp}{{\mathbb{S}^{p\times p}}}
\newcommand{\Sdd}{{\mathbb{S}^{d\times d}}}

\newcommand{\Rnn}{{\mathbb{R}^{n\times n}}}
\newcommand{\Rnp}{{\mathbb{R}^{n\times p}}}
\newcommand{\Rdp}{{\mathbb{R}^{d\times p}}}

\newcommand{\Rpp}{\mathbb{R}^{p\times p}}

\newcommand{\Rp}{{\mathbb{R}^{p}}}

\newcommand{\Rdd}{{\mathbb{R}^{d\times d}}}

\newcommand{\reals}{{\mathbb{R}}}

\newcommand{\Rn}{{\mathbb{R}^n}}
\newcommand{\Rm}{{\mathbb{R}^m}}

\newcommand{\grad}{\mathrm{grad}}
\newcommand{\Hess}{\mathrm{Hess}}
\newcommand{\vecc}{\mathrm{vec}}

\newcommand{\diag}{\mathrm{diag}}
\newcommand{\D}{\mathrm{D}}

\newcommand{\calC}{\mathcal{C}}
\newcommand{\calE}{\mathcal{E}}
\newcommand{\calF}{\mathcal{F}}
\newcommand{\calL}{\mathcal{L}}

\newcommand{\calH}{\mathcal{H}}
\newcommand{\calO}{\mathcal{O}}
\newcommand{\rank}{\operatorname{rank}}

\newcommand{\frobnormbig}[1]{\big\|{#1}\big\|_\mathrm{F}}
\newcommand{\opnormsmall}[1]{\|{#1}\|_\mathrm{op}}
\newcommand{\frobnormsmall}[1]{\|{#1}\|_\mathrm{F}}
\newcommand{\smallfrobnorm}[1]{\|{#1}\|_\mathrm{F}}

\newcommand{\sqfrobnormsmall}[1]{\frobnormsmall{#1}^2}
\newcommand{\sqfrobnormbig}[1]{\frobnormbig{#1}^2}

\newcommand{\frobnorm}[2][F]{\left\|{#2}\right\|_\mathrm{#1}}
\newcommand{\sqfrobnorm}[2][F]{\frobnorm[#1]{#2}^2}

\newcommand{\floor}[1]{\lfloor #1 \rfloor}

\newcommand{\lambdamin}{\lambda_\mathrm{min}}

\newcommand{\ith}{$i$th }
\newcommand{\jth}{$j$th }

\newtheorem{theorem}{Theorem}[section]
\newtheorem{lemma}[theorem]{Lemma}
\newtheorem{proposition}[theorem]{Proposition}
\newtheorem{corollary}[theorem]{Corollary}

\newtheorem{definition}{Definition}[section]
\newtheorem{remark}[theorem]{Remark}